\documentclass[12pt]{amsart}
\usepackage{mathrsfs}
\usepackage{mathabx}
\usepackage{enumerate}
\usepackage{skak}
\usepackage{bbm}
\usepackage{skull}
\usepackage{mathrsfs, mathtools}
\usepackage{stmaryrd}
\usepackage{amsmath,amssymb,amsfonts, mathabx}
\usepackage[all,cmtip]{xy}
\usepackage{a4wide}
\usepackage[mathscr]{eucal}
\usepackage[toc,page]{appendix}
\usepackage{verbatim}
\usepackage{dsfont,color}

\usepackage{extarrow}
\usepackage{enumerate}
\usepackage{fontenc}
\usepackage[T1]{fontenc}
\usepackage{graphicx}
\usepackage[colorlinks=true, linkcolor=blue]{hyperref}
\usepackage{latexsym}
\usepackage{mathrsfs}
\usepackage{mathtext}
\usepackage{tikz}
\usepackage{textcomp}
\usepackage{upgreek}
\usepackage{xy}

\usetikzlibrary{arrows}
\usetikzlibrary{matrix}
\usetikzlibrary{shapes}
\usetikzlibrary{snakes}
\usetikzlibrary{matrix}

\DeclareMathOperator{\Add}{\textup{Add}}
\DeclareMathOperator{\Alg}{\textup{Alg}}
\DeclareMathOperator{\Ann}{\textup{Ann}}
\DeclareMathOperator{\Arr}{\textup{Arr}}
\DeclareMathOperator{\Art}{\textup{Art}}
\DeclareMathOperator{\Ass}{\textup{Ass}}
\DeclareMathOperator{\Autsh}{\underline{\textup{Aut}}}
\DeclareMathOperator{\Bi}{\textup{B}}
\DeclareMathOperator{\CAdd}{\textup{CAdd}}
\DeclareMathOperator{\CAlg}{\textup{CAlg}}
\DeclareMathOperator{\CMon}{\textup{CMon}}
\DeclareMathOperator{\CPMon}{\textup{CPMon}}
\DeclareMathOperator{\CRings}{\textup{CRings}}
\DeclareMathOperator{\CSMon}{\textup{CSMon}}
\DeclareMathOperator{\CaCl}{\textup{CaCl}}
\DeclareMathOperator{\Cart}{\textup{Cart}}
\DeclareMathOperator{\Cl}{\textup{Cl}}
\DeclareMathOperator{\Coh}{\textup{Coh}}
\DeclareMathOperator{\Coker}{\textup{Coker}}
\DeclareMathOperator{\Der}{\textup{Der}}
\DeclareMathOperator{\End}{\textup{End}}
\DeclareMathOperator{\Endsh}{\underline{\textup{End}}}
\DeclareMathOperator{\Ext}{\textup{Ext}}
\DeclareMathOperator{\Extsh}{\underline{\textup{Ext}}}
\DeclareMathOperator{\FAdd}{\textup{FAdd}}
\DeclareMathOperator{\FCoh}{\textup{FCoh}}
\DeclareMathOperator{\FGrad}{\textup{FGrad}}
\DeclareMathOperator{\FLoc}{\textup{FLoc}}
\DeclareMathOperator{\FMod}{\textup{FMod}}
\DeclareMathOperator{\FPMon}{\textup{FPMon}}
\DeclareMathOperator{\FRep}{\textup{FRep}}
\DeclareMathOperator{\FSMon}{\textup{FSMon}}
\DeclareMathOperator{\FVect}{\textup{FVect}}
\DeclareMathOperator{\Fibr}{\textup{Fibr}}
\DeclareMathOperator{\Fix}{\textup{Fix}}
\DeclareMathOperator{\Fl}{\textup{Fl}}
\DeclareMathOperator{\Fr}{\textup{Fr}}
\DeclareMathOperator{\Funct}{\textup{Funct}}
\DeclareMathOperator{\GAlg}{\textup{GAlg}}
\DeclareMathOperator{\GExt}{\textup{GExt}}
\DeclareMathOperator{\GHom}{\textup{GHom}}
\DeclareMathOperator{\GL}{\textup{GL}}
\DeclareMathOperator{\GMod}{\textup{GMod}}
\DeclareMathOperator{\GRis}{\textup{GRis}}
\DeclareMathOperator{\GRiv}{\textup{GRiv}}
\DeclareMathOperator{\Gl}{\textup{Gl}}
\DeclareMathOperator{\Grad}{\textup{Grad}}
\DeclareMathOperator{\Hilb}{\textup{Hilb}}
\DeclareMathOperator{\Hl}{\textup{H}}
\DeclareMathOperator{\Homsh}{\underline{\textup{Hom}}}
\DeclareMathOperator{\ISym}{\textup{Sym}^*}
\DeclareMathOperator{\Imm}{\textup{Im}}
\DeclareMathOperator{\Irr}{\textup{Irr}}
\DeclareMathOperator{\Iso}{\textup{Iso}}
\DeclareMathOperator{\Isosh}{\underline{\textup{Iso}}}
\DeclareMathOperator{\LAdd}{\textup{LAdd}}
\DeclareMathOperator{\LAlg}{\textup{LAlg}}
\DeclareMathOperator{\LMon}{\textup{LMon}}
\DeclareMathOperator{\LPMon}{\textup{LPMon}}
\DeclareMathOperator{\LRings}{\textup{LRings}}
\DeclareMathOperator{\LSMon}{\textup{LSMon}}
\DeclareMathOperator{\Left}{\textup{L}}
\DeclareMathOperator{\Lex}{\textup{Lex}}
\DeclareMathOperator{\ML}{\textup{ML}}
\DeclareMathOperator{\MLex}{\textup{MLex}}
\DeclareMathOperator{\Mon}{\textup{Mon}}
\DeclareMathOperator{\Ob}{\textup{Ob}}
\DeclareMathOperator{\Obj}{\textup{Obj}}
\DeclareMathOperator{\PDiv}{\textup{PDiv}}
\DeclareMathOperator{\PGL}{\textup{PGL}}
\DeclareMathOperator{\PML}{\textup{PML}}
\DeclareMathOperator{\PMLex}{\textup{PMLex}}
\DeclareMathOperator{\PMon}{\textup{PMon}}
\DeclareMathOperator{\Picsh}{\underline{\textup{Pic}}}
\DeclareMathOperator{\Pro}{\textup{Pro}}
\DeclareMathOperator{\Proj}{\textup{Proj}}
\DeclareMathOperator{\QAdd}{\textup{QAdd}}
\DeclareMathOperator{\QAlg}{\textup{QAlg}}
\DeclareMathOperator{\QCoh}{\textup{QCoh}}
\DeclareMathOperator{\QMon}{\textup{QMon}}
\DeclareMathOperator{\QPMon}{\textup{QPMon}}
\DeclareMathOperator{\QRings}{\textup{QRings}}
\DeclareMathOperator{\QSMon}{\textup{QSMon}}
\DeclareMathOperator{\Rings}{\textup{Rings}}
\DeclareMathOperator{\Riv}{\textup{Riv}}
\DeclareMathOperator{\SFibr}{\textup{SFibr}}
\DeclareMathOperator{\SMLex}{\textup{SMLex}}
\DeclareMathOperator{\SMex}{\textup{SMex}}
\DeclareMathOperator{\SMon}{\textup{SMon}}
\DeclareMathOperator{\SchI}{\textup{SchI}}
\DeclareMathOperator{\Sh}{\textup{Sh}}
\DeclareMathOperator{\Soc}{\textup{Soc}}
\DeclareMathOperator{\Specsh}{\underline{\textup{Spec}}}
\DeclareMathOperator{\Stab}{\textup{Stab}}
\DeclareMathOperator{\Supp}{\textup{Supp}}
\DeclareMathOperator{\Sym}{\textup{Sym}}
\DeclareMathOperator{\TMod}{\textup{TMod}}
\DeclareMathOperator{\Top}{\textup{Top}}
\DeclareMathOperator{\Tor}{\textup{Tor}}
\DeclareMathOperator{\alt}{\textup{ht}}
\DeclareMathOperator{\car}{\textup{char}}
\DeclareMathOperator{\degtr}{\textup{degtr}}
\DeclareMathOperator{\depth}{\textup{depth}}
\DeclareMathOperator{\divis}{\textup{div}}
\DeclareMathOperator{\et}{\textup{et}}
\DeclareMathOperator{\ffpSch}{\textup{ffpSch}}
\DeclareMathOperator{\h}{\textup{h}}
\DeclareMathOperator{\ilim}{\displaystyle{\lim_{\longrightarrow}}}
\DeclareMathOperator{\indim}{\textup{inj dim}}
\DeclareMathOperator{\lf}{\textup{LF}}
\DeclareMathOperator{\op}{\textup{op}}
\DeclareMathOperator{\ord}{\textup{ord}}
\DeclareMathOperator{\pd}{\textup{pd}}
\DeclareMathOperator{\plim}{\displaystyle{\lim_{\longleftarrow}}}
\DeclareMathOperator{\pr}{\textup{pr}}
\DeclareMathOperator{\pt}{\textup{pt}}
\DeclareMathOperator{\rk}{\textup{rk}}
\DeclareMathOperator{\tr}{\textup{tr}}
\DeclareMathOperator{\type}{\textup{r}}
\DeclareMathOperator*{\colim}{\textup{colim}}
\usepackage[colorinlistoftodos]{todonotes}

\theoremstyle{plain}
\newtheorem{thm}{Theorem}[section]
\newtheorem{lem}[thm]{Lemma}

\newtheorem{prop}[thm]{Proposition}

\theoremstyle{definition}

\newtheorem{defn}[thm]{Definition} 
\newtheorem{ex}[thm]{Example}

\newtheorem{rmk}[thm]{Remark}

\numberwithin{thm}{section}
\newcounter{x}

\renewcommand{\et}{\textup{\'et}}

\newcommand{\red}{{\rm red}}

\newcommand{\EF}{{\rm EFin}}

\newcommand{\codim}{{\rm codim}}

\newcommand{\Pic}{{\rm Pic}}

\newcommand{\Div}{{\rm Div}}
\newcommand{\Strat}{{\rm Str}}
\newcommand{\Hom}{{\rm Hom}}

\newcommand{\Loc}{{\rm Loc}}

\newcommand{\Spec}{{\rm Spec \,}}
\newcommand{\Crys}{{\rm Crys}}
\newcommand{\Aff}{{\rm Aff}}

\newcommand{\Gal}{{\rm Gal}}
\newcommand{\Fdiv}{{\rm Fdiv}}
\newcommand{\FDiv}{{\rm Fdiv}}
\newcommand{\Ker}{{\rm Ker}}
\newcommand{\Aut}{{\rm Aut}}

\newcommand{\pp}{{\mathfrak{p}}}

\newcommand{\sB}{{\mathcal B}}
\newcommand{\sC}{{\mathcal C}}

\newcommand{\sE}{{\mathcal E}}
\newcommand{\sF}{{\mathcal F}}
\newcommand{\sG}{{\mathcal G}}

\newcommand{\sO}{{\mathcal O}}

\newcommand{\sU}{{\mathcal U}}

\newcommand{\sX}{{\mathcal X}}
\newcommand{\sY}{{\mathcal Y}}

\newcommand{\A}{{\mathbb A}}
\newcommand{\B}{{\mathbb B}}
\newcommand{\C}{{\mathbb C}}

\newcommand{\E}{{\mathbb E}}
\newcommand{\F}{{\mathbb F}}
\newcommand{\G}{{\mathbb G}}

\newcommand{\M}{{\mathbb M}}
\newcommand{\N}{{\mathbb N}}

\newcommand{\Q}{{\mathbb Q}}
\newcommand{\R}{{\mathbb R}}

\newcommand{\Z}{{\mathbb Z}}

\newcommand{\NN}{\textup{N}}

\newcommand{\Mod}{\text{\sf Mod}}

\newcommand{\Vect}{\text{\sf Vect}}
\newcommand{\Rep}{\text{\sf Rep}}
\newcommand{\id}{{\rm id\hspace{.1ex}}}

\newcommand{\ind}{{\text{\sf ind}\hspace{.1ex}}}
\newcommand{\Cov}{{\mathcal Cov}}

\newcommand{\eet}{\textup{\'et}}

\theoremstyle{plain}
\newtheorem{thmI}{Theorem}
\newtheorem{thmII}{Theorem}
\newtheorem{corI}{Corollary}
\newtheorem{corII}{Corollary}
\newtheorem{corIII}{Corollary}

\begin{document}
\title{$F$-divided sheaves trivialized by dominant maps are essentially finite}

\author{Fabio Tonini,  Lei Zhang }
\address{ Fabio Tonini\\
    Freie Universit\"at Berlin\\
    FB Mathematik und Informatik\\
    Arnimallee 3\\ Zimmer 112A\\
    14195 Berlin\\ Deutschland }
\email{tonini@math.hu-berlin.de}
 \address{ Lei Zhang\\
    Freie Universit\"at Berlin\\
    FB Mathematik und Informatik\\
    Arnimallee 3\\ Zimmer 112A\\
    14195 Berlin\\ Deutschland }
\email{l.zhang@fu-berlin.de}

\thanks{This work was supported by the European Research Council (ERC) Advanced Grant 0419744101 and the Einstein Foundation}
\date{\today}

\global\long\def\A{\mathbb{A}}

\global\long\def\Ab{(\textup{Ab})}

\global\long\def\C{\mathbb{C}}

\global\long\def\Cat{(\textup{cat})}

\global\long\def\Di#1{\textup{D}(#1)}

\global\long\def\E{\mathcal{E}}

\global\long\def\F{\mathbb{F}}

\global\long\def\GCov{G\textup{-Cov}}

\global\long\def\Gcat{(\textup{Galois cat})}

\global\long\def\Gfsets#1{#1\textup{-fsets}}

\global\long\def\Gm{\mathbb{G}_{m}}

\global\long\def\GrCov#1{\textup{D}(#1)\textup{-Cov}}

\global\long\def\Grp{(\textup{Grps})}

\global\long\def\Gsets#1{(#1\textup{-sets})}

\global\long\def\HCov{H\textup{-Cov}}

\global\long\def\MCov{\textup{D}(M)\textup{-Cov}}

\global\long\def\MHilb{M\textup{-Hilb}}

\global\long\def\N{\mathbb{N}}

\global\long\def\PGor{\textup{PGor}}

\global\long\def\PGrp{(\textup{Profinite Grp})}

\global\long\def\PP{\mathbb{P}}

\global\long\def\Pj{\mathbb{P}}

\global\long\def\Q{\mathbb{Q}}

\global\long\def\RCov#1{#1\textup{-Cov}}

\global\long\def\RR{\mathbb{R}}

\global\long\def\Sch{\textup{Sch}}

\global\long\def\WW{\textup{W}}

\global\long\def\Z{\mathbb{Z}}

\global\long\def\acts{\curvearrowright}

\global\long\def\alA{\mathscr{A}}

\global\long\def\alB{\mathscr{B}}

\global\long\def\arr{\longrightarrow}

\global\long\def\arrdi#1{\xlongrightarrow{#1}}

\global\long\def\catC{\mathscr{C}}

\global\long\def\catD{\mathscr{D}}

\global\long\def\catF{\mathscr{F}}

\global\long\def\catG{\mathscr{G}}

\global\long\def\comma{,\ }

\global\long\def\covU{\mathcal{U}}

\global\long\def\covV{\mathcal{V}}

\global\long\def\covW{\mathcal{W}}

\global\long\def\duale#1{{#1}^{\vee}}

\global\long\def\fasc#1{\widetilde{#1}}

\global\long\def\fsets{(\textup{f-sets})}

\global\long\def\iL{r\mathscr{L}}

\global\long\def\id{\textup{id}}

\global\long\def\la{\langle}

\global\long\def\odi#1{\mathcal{O}_{#1}}

\global\long\def\ra{\rangle}

\global\long\def\set{(\textup{Sets})}

\global\long\def\sets{(\textup{Sets})}

\global\long\def\shA{\mathcal{A}}

\global\long\def\shB{\mathcal{B}}

\global\long\def\shC{\mathcal{C}}

\global\long\def\shD{\mathcal{D}}

\global\long\def\shE{\mathcal{E}}

\global\long\def\shF{\mathcal{F}}

\global\long\def\shG{\mathcal{G}}

\global\long\def\shH{\mathcal{H}}

\global\long\def\shI{\mathcal{I}}

\global\long\def\shJ{\mathcal{J}}

\global\long\def\shK{\mathcal{K}}

\global\long\def\shL{\mathcal{L}}

\global\long\def\shM{\mathcal{M}}

\global\long\def\shN{\mathcal{N}}

\global\long\def\shO{\mathcal{O}}

\global\long\def\shP{\mathcal{P}}

\global\long\def\shQ{\mathcal{Q}}

\global\long\def\shR{\mathcal{R}}

\global\long\def\shS{\mathcal{S}}

\global\long\def\shT{\mathcal{T}}

\global\long\def\shU{\mathcal{U}}

\global\long\def\shV{\mathcal{V}}

\global\long\def\shW{\mathcal{W}}

\global\long\def\shX{\mathcal{X}}

\global\long\def\shY{\mathcal{Y}}

\global\long\def\shZ{\mathcal{Z}}

\global\long\def\st{\ | \ }

\global\long\def\stA{\mathcal{A}}

\global\long\def\stB{\mathcal{B}}

\global\long\def\stC{\mathcal{C}}

\global\long\def\stD{\mathcal{D}}

\global\long\def\stE{\mathcal{E}}

\global\long\def\stF{\mathcal{F}}

\global\long\def\stG{\mathcal{G}}

\global\long\def\stH{\mathcal{H}}

\global\long\def\stI{\mathcal{I}}

\global\long\def\stJ{\mathcal{J}}

\global\long\def\stK{\mathcal{K}}

\global\long\def\stL{\mathcal{L}}

\global\long\def\stM{\mathcal{M}}

\global\long\def\stN{\mathcal{N}}

\global\long\def\stO{\mathcal{O}}

\global\long\def\stP{\mathcal{P}}

\global\long\def\stQ{\mathcal{Q}}

\global\long\def\stR{\mathcal{R}}

\global\long\def\stS{\mathcal{S}}

\global\long\def\stT{\mathcal{T}}

\global\long\def\stU{\mathcal{U}}

\global\long\def\stV{\mathcal{V}}

\global\long\def\stW{\mathcal{W}}

\global\long\def\stX{\mathcal{X}}

\global\long\def\stY{\mathcal{Y}}

\global\long\def\stZ{\mathcal{Z}}

\global\long\def\then{\ \Longrightarrow\ }

\global\long\def\L{\textup{L}}

\global\long\def\l{\textup{l}}

\makeatletter 
\providecommand\@dotsep{5} 
\makeatother 

\setcounter{section}{0}
\maketitle
\begin{abstract}
By a result of Biswas and Dos Santos, on a smooth and projective variety over an algebraically closed field, a vector bundle trivialized by a proper and surjective map is essentially finite, that is it corresponds to a representation of the Nori fundamental group scheme.  In this paper we obtain similar results  for non-proper non-smooth algebraic stacks over arbitrary fields of characteristic $p>0$. As by-product we have the following partial generalization of the Biswas-Dos Santos' result in positive characteristic: on a pseudo-proper and inflexible stack of finite type over $k$ a vector bundle which is trivialized by a proper and flat map is essentially finite.
\end{abstract}

\section*{Introduction}
Let $X$ be a smooth and projective variety over an algebraically closed field $k$ with a rational point $x\in X(k)$. Then the    \'etale universal cover $\widetilde{X}_x$ of $X$ is an fpqc torsor under the \'etale fundamental group $\pi_1^{\textup{\'et}}(X,x)$ (viewing as a pro-finite group scheme over $k$). Thus any representation of $\pi_1^{\textup{\'et}}(X,x)$ on a finite dimensional $k$-vector space $V$ will give rise to a vector bundle  $\widetilde{X}_x\times^{\pi_1^{\textup{\'et}}(X,x)}V$ on $X$ by fpqc descent. This defines a functor from the category of finite dimensional $\pi_1^{\textup{\'et}}(X,x)$-representations to the category of vector bundles on $X$.  It is Lange and Stuhler who first observed, in \cite[1.2]{LS}, that a vector bundle on $X$ is in the essential image of this functor if and only if the vector bundle is trivialized by a finite \'etale cover of $X$.

A similar result holds also for the Nori fundamental group scheme, introduced in \cite{Nori}. Let $X$ be a proper, geometrically connected and geometrically reduced scheme over a field $k$ equipped with a rational point $x\in X(k)$. Nori defined the notion of \textit{essentially finite} vector bundles on $X$ and proved that they form a Tannakian category over $k$, whose associated group scheme is called the Nori fundamental group scheme of $(X,x)$.  Directly  
from Nori's definition one has that a vector bundle on $X$ is essentially finite if and only if it is trivialized by a torsor over $X$ under some finite $k$-group scheme. However, torsors under finite group schemes only correspond to Galois covers in the \'etale case. Therefore one should expect more, for instance that vector bundles trivialized by flat and finite maps are essentially finite. 
In \cite{BDS} and \cite{AM} this question has been answered for normal, projective varieties over an algebraically closed field: in this case a vector bundle is essentially finite if and only if it is trivialized by a proper and surjective morphism.

Meanwhile Nori's construction and the notion of essentially finite vector bundles has been largely extended. In \cite{BV} N. Borne and A. Vistoli introduced the notion of Nori fundamental gerbe for a fibered category $\stX$ over a field $k$ via a universal property and also an abstract notion of essentially finite object in an additive and monoidal category. The fibered categories over $k$ which are inflexible over $k$ (see \ref{inflexible}) are the ones admitting a Nori gerbe. For instance if $\stX$ is a reduced stack of finite type then it is inflexible if and only if $k$ is integrally closed in the ring $\Hl^0(\odi\stX)$ (see \ref{inflexible properties}). A fibered category $\stX$ over $k$ is pseudo-proper over $k$ if for all $E\in \Vect(\stX)$ the $k$-vector space $\Hl^0(E)$ is finite dimensional. Examples of pseudo-proper stacks are of course proper algebraic stacks over $k$ but also arbitrary affine gerbes. The previous result of Nori is generalized by showing that for a pseudo-proper and inflexible fibered category $\stX$ over $k$ the category of essentially finite vector bundles on $\stX$ is a $k$-Tannakian category and its associated gerbe is the Nori fundamental gerbe (see \cite[Theorem 7.9]{BV}). Applying this to an affine gerbe, they also concluded that, in a $k$-Tannakian category, the full subcategory of essentially finite objects is the $k$-Tannakian subcategory of objects having finite monodromy gerbe, which is the gerbe corresponding to the $k$-Tannakian subcategory generated by this object (see \cite[Corollary 7.10]{BV}).
%

Based on  \cite{BV} and \cite{EH}, we also get in \cite{TZ} a Tannakian description of the Nori fundamental gerbe for not necessarily pseudo-proper fibered categories using the language of stratified sheaves and $F$-divided sheaves. The theory applies to all reduced and inflexible algebraic stacks of finite type over $k$. In particular in positive characteristic we studied Tannakian categories of $F$-divided sheaves $\Fdiv(\stX/k)$ and $\Fdiv_{\infty}(\sX/k)$ and prove that their essentially finite objects are the representations of the Nori \'etale fundamental gerbe and Nori fundamental gerbe of $\stX/k$ respectively.

In this paper we consider Lange-Stuhler or Biswas-Dos Santos style theorem in the non pseudo-proper settings. In what follows we assume that the ground field $k$ has positive characteristic. An object in a additive monoidal category is called trivial or free if it is isomorphic to a finite direct sum of copies of the unit object. Given a map of algebraic stacks $f\colon \stU\arr \stX$ we denote by $\FDiv(\stX/k)_f$ the full subcategory of $\FDiv(\stX/k)$ of objects trivialized by $f$, that is whose pullback along $f$ is free in $\FDiv(\stU/k)$. When $\stX$ is reduced and inflexible we denote by $\FDiv_\infty(\stX/k)_f$ the sub-Tannakian category of $\FDiv_\infty(\stX/k)$ generated by the objects trivialized by $f$. Here are our main results:

\begin{thmI}\label{thmI}
Let $\sX$ be a geometrically connected (resp. reduced and inflexible) algebraic stack of finite type over $k$. If $\E\in \FDiv(\stX/k)$ (resp. $\E\in \Fdiv_\infty(\stX/k)$) is an essentially finite object then there exists a surjective finite and \'etale (resp. finite and flat)  map $f\colon \stU\arr\stX$ trivializing $\E$. Conversely let $f\colon \stU\arr\stX$ be a flat and surjective map of algebraic  stacks of finite type over $k$. Then $\Fdiv(\stX/k)_f$ (resp. $\FDiv_\infty(\stX/k)_f$) is a Tannakian subcategory and, if one of the conditions below is satisfied, its associated gerbe is finite and \'etale (resp. profinite):
\begin{enumerate}
\item the map $f$ is proper;
\item the map $f$ is geometrically connected, in which case $\FDiv(\stX/k)_f=\Vect(k)$;
\item the stack $\sU$ is geometrically irreducible.
\end{enumerate}
\end{thmI}

In this generality we are unable to prove the result for arbitrary flat and surjective maps $f$, although we believe this should be true. On the other hand adding some more regularity we can even drop the flatness hypothesis:

\begin{thmII}\label{thmII}
Let $\sX$ be a geometrically unibranch (e.g. normal, see Section \ref{Unibranch}) algebraic stack of finite type over $k$. If $\stX$ is geometrically connected (resp. reduced and inflexible) then an object in $\Fdiv(\stX/k)$ (resp. $\Fdiv_{\infty}(\sX/k)$)  is essentially finite if and only if it is trivialized by a dominant morphism of finite type from an algebraic stack.
\end{thmII}

By lifting vector bundles to objects of $\FDiv_\infty(\stX/k)$ we are able to deduce from Theorem \ref{thmI} this partial generalization of \cite{BDS} and \cite{AM}:

\begin{corI}\label{corI}
 Let $\stX$ be a pseudo-proper and inflexible algebraic stack $\stX$ of finite type over $k$. Then a vector bundle on $\stX$ is essentially finite if and only if it is trivialized by a proper and flat map $f\colon \stU\arr\stX$ from an algebraic stack such that $\Hl^0(\odi{\stU\times_\stX \stU})$ is a finite $k$-algebra.
\end{corI}

Essentially we drop the smothness assumption but we require the trivialization to be flat. The finiteness condition is needed in this generality because $\stU$ or $\stU\times_\stX\stU$ may not be pseudo-proper. For instance asking that the space of global sections of all coherent sheaves on $\stX$ are finite dimensional solve this issue.

Another application of the theory and, more precisely, of Corollary \ref{corI} is the following. In \cite{BV2} Borne and Vistoli introduce the virtually unipotent fundamental gerbe of a fibred category $\stX$ over $k$, denoted by $\stX\arr \Pi^{\textup{VU}}_{\stX/k}$: it is defined as a pro virtually unipotent affine gerbe such that any map $\stX\arr \Gamma$ to a virtually unipotent gerbe factors uniquely through $\Pi^{\textup{VU}}_{\stX/k}$ (see \ref{virtually unipotent} and \ref{virtually unipotent fundamental gerbe}). If $k$ has positive characteristic, and if $\stX$ is pseudo-proper, geometrically reduced, geometrically connected over $k$ and has an atlas from a reduced Noetherian scheme, then the representations of $\Pi^{\textup{VU}}_{\stX/k}$ have a Tannakian interpretation: they correspond, up to Frobenius, to subsequent extensions of essentially finite vector bundles (see \cite[Section 10, Theorem 10.7]{BV2}). We deduce the following:

\begin{corII}\label{corII}
Let $\stX$ be a pseudo-proper, geometrically reduced and geometrically connected algebraic stack of finite type over a field of positive characteristic $k$ such that $\dim_k \Hl^1(\stX,E)<\infty$ for all vector bundles $E$ on $\stX$. Then the virtually unipotent gerbe $\Pi^{\textup{VU}}_{\stX/k}$ and the Nori fundamental gerbe $\Pi^\NN_{\stX/k}$ concide.
\end{corII}
The above result holds only in positive characteristic. Corollary \ref{corII} fails for a stack like $\sB_k \G_a$, so that the conditions on $\Hl^1$ is also necessary: it ensures that all $\G_a$-torsors of $\stX$ and of its covers comes from a finite subgroup of $\G_a$ (see \ref{additive torsors}).  

Finally as application of Theorem \ref{thmII} we prove the following:

\begin{corIII}\label{corIII}
Let $\Gamma$ be a gerbe of finite type over $k$. Then all objects of $\FDiv(\Gamma/k)$ and $\FDiv_\infty(\Gamma/k)$ are essentially finite. In particular
\[
\FDiv(\Gamma/k)\simeq \Vect(\Gamma_\et) \text{ and } \Fdiv_\infty(\Gamma/k) \simeq \Vect(\hat\Gamma)
\]
where $\Gamma_\et$ and $\hat\Gamma$ are the pro\'etale and profinite quotients respectively.
\end{corIII}

The paper is divided as follows. In the first section we recall various properties and definitions, like of $F$-divided sheaves and Nori gerbes, and state the results we will use about them. The second section is dedicated to the proof of Theorem \ref{thmI} and Corollaries \ref{corI} and \ref{corII}, while the third one to the proof of Theorem \ref{thmII} and Corollary \ref{corIII}. Finally in the last appendices we discuss the notion of geometrically unibranch algebraic stacks and the behaviour of $\FDiv(-/k)$ under finite field extensions.

\section*{Notation}
A category fibered in groupoid $\stX$ over a ring $R$ is a category fibered in groupoid over the category of affine $R$-schemes $\Aff/R$. By an fpqc covering $\stU\arr \stX$ we mean a map of fibered categories which is respesented by fpqc coverings of algebraic spaces. If $\stX$ is an algebraic stack we denote by $\stX_\red$ its reduction.

If $\stX$ is a fibered category over $\F_p$ one can always define a Frobenius functor $F_\stX\colon\stX\arr\stX$ and, if $\stX$ is defined also over a field extension $k/\F_p$, a relative Frobenius functor $\stX\arr\stX^{(i,k)}$ for $i\in \N$. Here and in the rest of the paper by $\stX^{(i,k)}$ we mean the base change of $\stX$ over the $F_k^i\colon k\arr k$. When $k$ is clear from the context we will simply write $\stX^{(i)}$. Please refer to section \emph{Notations and Conventions} of \cite{TZ} for details.

In this paper we will freely talk about affine gerbes over a field (often improperly called just gerbes) and Tannakian categories and use their properties. Please refer to \cite[Appendix B]{TZ} for details. If $\shC$ is a $k$-Tannakian category then its associated affine gerbe over $k$ is denoted by $\Pi_\shC$: if $A$ is a $k$-algebra then $\Pi_\shC(A)$ is the category of $k$-linear, monoidal and exact functors $\shC\arr \Vect(A)$. 

If $\shC$ is a monoidal category we will often denote by $\mathds{1}_\shC$ or simply $\mathds{1}$ the unit object of $\shC$. An object in $\shC$ is called trivial or free if it is isomorphic to $\mathds{1}_\shC^{\oplus m}$ for some $m\in \N$.


\section*{Acknowledgement}
 We would like to thank B. Bhatt, N. Borne, H. Esnault, M. Olsson, M. Romagny and A. Vistoli for helpful conversations and suggestions received.

\section{Preliminaries}

 Everything in this section is either contained in or easy consequences of  results in  \cite{BV} and \cite{TZ}. We collect some important theorems and definitions here for the convenience of the reader.

\begin{defn}
Let $k$ be a field of characteristic $p>0$ and let $\sX$ be a category fibered in groupoid over $k$. We define $\Fdiv(\sX/k)$ to be the category with:

 \textbf{Objects}: tuples $(E_i,\sigma_i)_{i\in \N}$, where $E_i$ is a vector bundle on $\sX^{(i)}$ and an isomorphism $\sigma_i: \phi_{i,i+1}^*E_{i+1}\arr E_i$, where $\phi_{i,i+1}: \sX^{(i)}\arr \sX^{(i+1)}$ is the relative Frobenius map of $\sX^{(i)}/k$;
 
  \textbf{Morphisms}: a morphism from  $(E_i,\sigma_i)_{i\in \N}$ to $(E_i',\sigma_i')_{i\in \N}$ is just a collection of morphisms $a_i:E_i\arr E_i'$ making the diagram $$\xymatrix{\phi_{i,i+1}^*E_{i+1}\ar[rr]^{\sigma_i}\ar[d]_{\phi_{i,i+1}^*(a_{i+1})}&& E_i\ar[d]^{a_{i}}\\ \phi_{i,i+1}^*E_{i+1}'\ar[rr]^{\sigma_i'}&& E_i'}$$ commutative.

Given a natural number $i\in \N$ we also define the category $\Fdiv_i(\stX/k)$ as follows:

\textbf{Objects}:  triples $(\sF,\sG,\lambda)$, where $\sF\in\Vect(\sX)$, $\sG=(\sG_i,\sigma_i)_{i\in\N}\in\Fdiv(\sX/k)$ and $\lambda: F^{i^*}\sF\xrightarrow{\cong}\sG_0$ is an isomorphism (here $F$ denotes the absolute Frobenius of $\sX$);

  \textbf{Morphisms}: a morphism from  $(\sF,\sG,\lambda)$ to $(\sF',\sG',\lambda')$ is just a pair of morphisms $\phi:\sF\arr \sF'$ and $\varphi: \sG\arr \sG'$ making the diagram $$\xymatrix{F^{i^*}\sF\ar[rr]^{\lambda}\ar[d]_{F^{i^*}(\phi)}&& \sG_0\ar[d]^{\varphi_0}\\ F^{i^*}\sF'\ar[rr]^{\lambda'}&& \sG_0'}$$ commutative.
  
The categories $\Fdiv_i(\sX/k)$ are additive and monoidal and they have a $k$-linear structor coming from the one on $\sF$. There are monoidal and $k$-linear functors $\Fdiv_{i}(\sX/k)\arr \Fdiv_{i+1}(\sX/k)$, $(\sF,\sG,\lambda)\mapsto (\sF,F^*\sG,F^*\lambda)$ and we define $\Fdiv_{\infty}(\sX/k):=\varinjlim_{i\in\N}\Fdiv_{i}(\sX/k)$. (See \cite[Definition 5.9]{TZ} for details.)
\end{defn}

\begin{rmk}
Note that if $k$ is a perfect field, then one can easily check that $\Fdiv_{i_0}(\sX/k)$ is equivalent to the category defined as follows:

\textbf{Objects}:  tuples $(E_i,\sigma_i)_{i\in \N}$, where $E_i$ is a vector bundle on $\sX^{(i)}$, $\sigma_i:\phi_{i,i+1}^*E_{i+1}\xrightarrow{\cong} E_i$ for all $i\geq 1$, and $\sigma_0: \phi_{-i_0,0}^*E_{0}\xrightarrow{\cong} \phi_{-i_0,1}^*E_1$.

  \textbf{Morphisms}: a morphism from  $(E_i,\sigma_i)_{i\in \N}$ to $(E_i',\sigma_i')_{i\in \N}$ is just a collection of morphisms $a_i:E_i\arr E_i'$  which are compatible with $\sigma_i$ and $\sigma_{i}'$.
  
  Using this equivalence one can deduce that if  $k$ is perfect and if $\sX$ admits a representable  fpqc covering from a reduced scheme $\sX$ then the natural transition functors  $\Fdiv_{i}(\sX/k)\arr \Fdiv_{i+1}(\sX/k)$ are fully faithful.
\end{rmk}

\begin{rmk}\label{insensible to nilpotent thickenning}
 An important property of $\FDiv$ that will be used in all the paper is that it is insensible to nilpotent thickenings, that is if $\stX'\arr \stX$ is a nilpotent closed immersion then $\FDiv(\stX/k)\arr \FDiv(\stX'/k)$ is an equivalence (see \cite[Lemma 6.18]{TZ}). In particular if $\stX$ is a Noetherian algebraic stack and $\stX_\red$ its reduction then $\FDiv(\stX/k)\simeq \FDiv(\stX_\red/k)$.
\end{rmk}

 \begin{defn}
 Given a $k$-algebra $A$ we set
 \[
 A_{\et,k}=\{a\in A\st \exists \text{ a separable polynomial } f\in k[x] \text{ s.t. } f(a)=0\}
 \]
 Alternatively $A_{\et,k}$ is the union of all $k$-subalgebras of $A$ which are finite and \'etale over $k$. When the base field is clear from the context we will simply write $A_{\et}$.
\end{defn}

\begin{rmk}\cite[Lemma 2.6]{TZ}
 Let $\stX$ be a quasi-compact and quasi-separated algebraic stack over $k$. Then $\stX\arr \Spec \Hl^0(\odi\stX)_\et$ is geometrically connected and, if $\stX$ is connected, $\Hl^0(\odi\stX)_\et$ is a field. Moreover $\stX$ is geometrically connected over $k$ if and only if $\Hl^0(\odi\stX)_\et=k$
\end{rmk}

 \begin{defn}\label{inflexible}\cite[Section 5]{BV}
 If $\stX$ is a category fibered in groupoid over $k$ the Nori fundamental gerbe (resp. \'etale Nori fundamental gerbe) of $\stX/k$ is a profinite (resp. pro-\'etale) gerbe $\Pi$ over $k$ together with a map $\stX\arr \Pi$ such that for all finite (resp. finite and \'etale) stacks $\Gamma$ over $k$ the pullback functor
 \[
 \Hom_k(\Pi,\Gamma)\arr \Hom_k(\stX,\Gamma)
 \]
 is an equivalence. If this gerbe exists it is unique up to a unique isomorphism and it will be denoted by $\Pi^\NN_{\stX/k}$ (resp. $\Pi^{\NN,\et}_{\stX/k}$) or by dropping the $/k$ if it is clear from the context.
 
 We call $\sX$ inflexible if it is non-empty and all maps to a finite stack over $k$ factors through a finite gerbe over $k$.
\end{defn}

\begin{rmk}\label{inflexible properties}
By \cite[Definition 5.7, pp. 13]{BV} $\sX$ admits a Nori fundamental gerbe if and only it is inflexible.
By \cite[Theorem 4.4]{TZ} if $\stX$ is reduced, quasi-compact and quasi-separeted then $\stX$ is inflexible if and only if $k$ is algebraically closed in $\Hl^0(\odi \stX)$. In particular if $\stX$ is geometrically connected and geometrically reduced then it is inflexible.

By \cite[Proposition 4.3]{TZ} if $\stX$ is quasi-compact and quasi-separeted then $\stX$
admits a Nori \'etale fundamental gerbe if and only if $\stX$ is geometrically connected over $k$.
\end{rmk}

\begin{defn}\label{essentially finite}\cite[Definition 7.7, pp. 21]{BV}
Let $\shC$ be an additive and monoidal category. An object $E\in\shC$ is called finite if there exist $f\neq g \in \N[X]$ polynomials with natural coefficients and an isomorphism $f(E)\simeq g(E)$, it is called essentially finite if it is a kernel of a map of finite objects of $\shC$. We denote by $\EF(\shC)$ the full subcategory of $\shC$ consisting of essentially finite objects.
\end{defn}


\begin{defn}  \cite[Definition 7.1, pp. 20]{BV} 
A category $\sX$ fibered in groupoid over a field $k$ is \textit{pseudo-proper} if it satisfies the following two conditions:
\begin{enumerate}
\item there exists a quasi-compact scheme $U$ and a morphism $U\arr X$ which is representable, faithfully flat, quasi-compact, and quasi-separated.
\item for all vector bundles $E$ on $X$ the $k$-vector space $\Hl^0(\sX,E)$ is finite dimensional.
\end{enumerate}
\end{defn}

\begin{ex}  \cite[Example 7.2, pp. 20]{BV} 
Examples of pseudo-proper fiber categories are proper schemes, finite stacks and affine gerbes.
\end{ex}

\begin{thm}\label{monodromy essentially finite} \cite[Theorem 7.9, Corollary 7.10, pp. 22]{BV} 
Let $\sX$ be an inflexible pseudo-proper fibered category over a field $k$. Then the pullback of $\stX\arr\Pi^\NN_{\stX/k}$ induces an equivalence $\Vect(\Pi_{\sX/k}^\NN)\arr \EF(\Vect(\sX))$.

Let $\sC$ be a Tannakian category. Then $\EF(\sC)$ is the Tannakian subcategory of $\shC$ of objects whose monodromy gerbe is finite.
\end{thm}

\begin{defn} \label{separably generated}
 A field extension $L/k$ is called separably generated (resp. separable) up to Frobenius if $\car k=0$ or $\car k=p>0$ and there exists $i\in \N$ such that $L^{p^i}$ is contained in a separably generated (resp. separable) field extension of $k$ (see \cite[\href{http://stacks.math.columbia.edu/tag/030I}{030I}]{SP}).
\end{defn}

Putting together results \cite[Lemma 2.6, Theorem 4.4, Theorem 5.8, Theorem 5.12, Theorem 6.17]{TZ} we deduce the following:

\begin{thm}\label{thm for Fdiv}
Let $\stX$ be a connected algebraic stack over $k$ with an fpqc covering $U\arr \stX$ from a Noetherian scheme $U$ such that, for all $u\in U$, $k(u)/k$ is separable up to Frobenius and set $L=\End_{\FDiv(\stX/k)}(\mathds{1})$. Then $\FDiv(\stX/k)\arr \Vect(\stX)$ is faithful, $L\subseteq \Hl^0(\odi\stX)$ is a field, $\Fdiv(\stX/k)$ is an $L$-Tannakian category and there is an equivalence of $L$-Tannakian categories $\Vect(\Pi^{\NN,\et}_{\stX/L})\simeq \EF(\Fdiv(\stX/k))$. If there exists a map $\Spec F\arr \stX$ where $F$ is a field separably generated up to Frobenius over $k$ then $L=\Hl^0(\odi\stX)_\et$, so that $L=k$ if and only if $\stX$ is geometrically connected over $k$.
 
 Assume $\stX$ reduced and set $L_i=\{x\in \Hl^0(\odi\stX)\st x^{p^i}\in L\}$ for $i\in \N$ and $L_\infty=\cup_i L_i$. Then $\Fdiv_i(\stX/k)$ is an $L_i$-Tannakian category for $i\in \N\bigcup\{\infty\}$ and there is an equivalence of $L_\infty$-Tannakian categories $\Vect(\Pi^\NN_{\stX/L_\infty})\simeq \EF(\Fdiv_\infty(\stX/k))$.  Moreover
 \[
 \EF(\FDiv_\infty(\stX/k))\simeq \varinjlim_i \EF(\FDiv_i(\stX/k))
 \]
 and $(\shF,\shG,\lambda)\in \FDiv_i(\stX/k)$ is essentialy finite if and only if $\shG$ is essentially finite in $\FDiv(\stX/k)$.
 If there exists a map $\Spec F\arr \stX$ where $F$ is a field separably generated up to Frobenius over $k$ then $L_\infty=k$ if and only if $\stX$ is inflexible over $k$.
\end{thm}
\begin{rmk}\label{cases for thm for Fdiv}
 A separably generated (up to Frobenius) field extension is separable (up to Frobenius) and all extensions of a perfect field are separable (see \cite[\href{http://stacks.math.columbia.edu/tag/05DT}{05DT}]{SP}). For instance finitely generated field extensions are separably generated up to Frobenius. Thus Theorem \ref{thm for Fdiv} apply when $\stX$ is a stack of finite type over $k$: if $\stX$ is geometrically connected over $k$ then $\FDiv(\stX/k)$ is a $k$-Tannakian category over $k$ and if $\stX$ is reduced and inflexible over $k$ then $\FDiv_\infty(\stX/k)$ is a $k$-Tannakian category too. 
\end{rmk}

\begin{defn}\cite[Definition 6.16]{BV2}\label{virtually unipotent}
 A \emph{virtually unipotent} group scheme over $k$ is an affine group scheme $G$ of finite type over $k$ such that the reduction of the connected component of $G\times_k {\overline k}$ is unipotent, where $\overline k$ is an algebraic closure of $k$.
 
 An affine gerbe $\Gamma$ over $k$ is virtually unipotent if $\Gamma\times_k \overline k \simeq \sB G$, where $G$ is a virtually unipotent group scheme over $\overline k$. A pro virtually unipotent gerbe is an affine gerbe projective limit of virtually unipotent gerbes.
 \end{defn}
 
 \begin{defn}\cite[Definition 5.6]{BV2}\label{virtually unipotent fundamental gerbe}
 If $\stX$ is a category fibered in groupoid over $k$ the \emph{virtually unipotent fundamental gerbe} of $\stX/k$ is a pro virtually unipotent gerbe $\Pi$ over $k$ together with a map $\stX\arr \Pi$ such that for all virtually unipotent gerbes $\Gamma$ over $k$ the pullback functor
 \[
 \Hom_k(\Pi,\Gamma)\arr \Hom_k(\stX,\Gamma)
 \]
 is an equivalence. If this gerbe exists it is unique up to a unique isomorphism and it will be denoted by $\Pi^{\textup{VU}}_{\stX/k}$ or by dropping the $/k$ if it is clear from the context.
 \end{defn}

 \begin{thm}\cite[Section 6.3, Theorem 7.1]{BV2}
  If $\stX$ is a quasi-compact, quasi-separated and geometrically reduced fibered category over $k$ such that $\Hl^0(\stX,\odi\stX)=k$ then $\stX$ admits a virtually unipotent fundamental gerbe.
 \end{thm}

\begin{defn}\cite[Definition 10.2, Definition 10.6]{BV2}\label{filtration vu}
 Let $\stX$ be a fibered category. A vector bundle $E$ on $\stX$ is called an \emph{extended essentially finite} sheaf if there is a filtration
 \[
 0=E_{r+1} \subseteq E_r \subseteq \cdots \subseteq E_1 \subseteq E_0 = E 
 \]
 in which all quotients $E_i/E_{i+1}$ are essentially finite vector bundles on $\stX$.
 
 Assume that $k$ has positive characteristic and let $F\colon \stX\arr\stX$ be the absolute Frobenius of $\stX$. A vector bundle $E$ on $\stX$ is called \emph{virtually unipotent} if there exists $m\in \N$ such that $F^{m*}E$ is an extended essentially finite vector bundle on $\stX$.
\end{defn}

\begin{thm}\cite[Theorem 10.7]{BV2}\label{Tannakian virtually unipotent}
 Let $\stX$ be a pseudo-proper, geometrically reduced and geometrically connected fibered category admitting an fpqc cover from a Noetherian reduced scheme. Then the pullback $\Vect(\Pi^{\textup{VU}}_{\stX/k})\arr \Vect(\stX)$ is an equivalence onto the full subcategory of $\Vect(\stX)$ of virtually unipotent sheaves.
\end{thm}

\section{Trivializations by flat and surjective maps}
We fix a field $k$ of positive characteristic.
The aim of this section is to prove Theorem \ref{thmI} and Corollaries \ref{corI} and \ref{corII}. We start by showing how to trivialize essentially finite $F$-divided sheaves.

\begin{prop}\label{voucher for finite stacks}
 Let $\Gamma$ be a finite stack over $k$. Then
 \[
 \FDiv(\Gamma/k)\simeq \Vect(\Gamma_\et)\text{ and }\Fdiv_\infty(\Gamma/k)\simeq \Vect(\Gamma)
 \]
\end{prop}
\begin{proof}
 By \cite[Lemma 3.6]{TZ} we can find an index $i>0$ and $2$-commutative diagrams
   \[
  \begin{tikzpicture}[xscale=3.0,yscale=-1.2]
    \node (A0_0) at (0, 0) {$\Gamma$};
    \node (A0_1) at (1, 0) {$\Gamma^{(i)}$};
    \node (A1_0) at (0, 1) {$\Gamma_\et$};
    \node (A1_1) at (1, 1) {$\Gamma_\et^{(i)}$};
    \path (A0_0) edge [->]node [auto] {$\scriptstyle{}$} (A0_1);
    \path (A1_0) edge [->]node [auto] {$\scriptstyle{c}$} (A0_1);
    \path (A1_0) edge [->]node [auto] {$\scriptstyle{\simeq}$} (A1_1);
    \path (A0_0) edge [->]node [auto] {$\scriptstyle{\alpha}$} (A1_0);
    \path (A0_1) edge [->]node [auto] {$\scriptstyle{}$} (A1_1);
  \end{tikzpicture}
  \]
The arrow $c$ provides an equivalence of categories between $\Gamma^{(\infty)}:=\varinjlim_{i\in\N}\Gamma^{(i)}$ and $\Gamma_\et^{(\infty)}:=\varinjlim_{i\in\N}\Gamma_\et^{(i)}$. Thus we have $\FDiv(\Gamma)=\Vect(\Gamma^{(\infty)})=\Vect(\Gamma_\et^{(\infty)})=\Vect(\Gamma_\et)$ (see \cite[\S 6.2, Definition 6.15 and Proposition 6.16]{TZ}).


Considering the factorization $F^i\colon \Gamma\arrdi\alpha \Gamma_{\et} \arrdi\beta \Gamma$, where $F$ is the absolute Frobenius of $\Gamma$, we get a functor $\Psi:\Vect(\Gamma)\arr \FDiv_\infty(\Gamma)$ mapping $V\in \Vect(\Gamma)$ to $(V,\beta^*V,v)\in \FDiv_i(\Gamma)$, where $v\colon \alpha^*(\beta^*V)\arr F^{i*}V$ is the canonical morphism. Here we identify the forgetful functor $\FDiv(\Gamma)\arr \Vect(\Gamma)$ with $\Vect(\Gamma_\et)\arrdi{\alpha^*}\Vect(\Gamma)$.
Since $\Psi$ is a section of the forgetful functor $\Phi\colon \FDiv_\infty(\Gamma)\arr \Vect(\Gamma)$, we conclude that $\Phi$ is essentially surjective. 

Let's prove that it is fully faithful and let $\chi=(V,W,u)$, $\chi'=(V',W',u')\in \FDiv_j(\Gamma)$ and $\lambda\colon V\arr V'$ be a map. 
Via $u,u'$ we obtain a map $\delta\colon\alpha^*W\arr\alpha^*W'$: an arrow $\chi\arr\chi'$ mapping to $\lambda$ is a pair $(\lambda,W\arrdi{\hat\delta}W')$ with $\alpha^*\hat\delta=\delta$.  Since $\alpha^*$ is faithful it follows that $\Phi$ is faithful. To prove that it is full, we must prove that there exists $l>0$ such that $F^{l*}\delta\colon \alpha^*F_{\Gamma_\et}^{l*}W\arr \alpha^*F_{\Gamma_\et}^{l*}W'$, where $F_{\Gamma_\et}$ is the absolute Frobenius of $\Gamma_\et$, comes from a map $\gamma\colon F_{\Gamma_\et}^{l*}W\arr F_{\Gamma_\et}^{l*}W'$. 
Since the composition $\Gamma_\et\arrdi\beta\Gamma\arrdi\alpha\Gamma_\et$ coincides with $F_{\Gamma_\et}^i$, it is enough to set $l=i$ and $\gamma=\beta^*\delta$.
\end{proof}

\begin{proof}[Proof of Theorem \ref{thmI}, first part]
We consider first the case of $\Fdiv(\stX)$. So let $\E\in \FDiv(\stX)$ which is essentially finite and denotes by $\Gamma$ the monodromy gerbe of $\E$. 
In particular there are maps $\stX\arr\Pi_{\Fdiv(\stX)}\arrdi\alpha\Gamma$ and $V\in \Vect(\Gamma)$ such that $\alpha^*V\simeq \E$.
The profinite quotient of $\Pi_{\FDiv(\stX)}$ is $\Pi^{\NN,\et}_\stX$ by \ref{thm for Fdiv} and therefore there is also a factorization $\alpha\colon \Pi_{\FDiv(\stX)}\arr\Pi^{\NN,\et}_\stX\arr\Gamma$. In particular $\Gamma$ is a finite and \'etale gerbe.
We have a $2$-commutative diagram
  \[
  \begin{tikzpicture}[xscale=2.0,yscale=-1.2]
    \node (A0_0) at (0, 0) {$\stX$};
    \node (A0_1) at (1, 0) {$\Pi_{\Fdiv(\stX)}$};
    \node (A0_2) at (2, 0) {$\Pi^{\NN,\et}_\stX$};
    \node (A1_0) at (0, 1) {$\Gamma$};
    \node (A1_1) at (1, 1) {$\Pi_{\FDiv(\Gamma)}$};
    \node (A1_2) at (2, 1) {$\Pi^{\NN,\et}_\Gamma$};
    \path (A0_1) edge [->,dashed]node [auto,swap] {$\scriptstyle{\alpha}$} (A1_0);
    \path (A0_0) edge [->]node [auto] {$\scriptstyle{}$} (A0_1);
    \path (A0_1) edge [->]node [auto] {$\scriptstyle{}$} (A0_2);
    \path (A1_0) edge [->]node [auto] {$\scriptstyle{a}$} (A1_1);
    \path (A1_1) edge [->]node [auto] {$\scriptstyle{}$} (A1_2);
    \path (A0_2) edge [->]node [auto] {$\scriptstyle{}$} (A1_2);
    \path (A0_0) edge [->]node [auto] {$\scriptstyle{}$} (A1_0);
    \path (A0_1) edge [->]node [auto] {$\scriptstyle{b}$} (A1_1);
  \end{tikzpicture}
  \]
and, by \ref{voucher for finite stacks}, the lower horizontal arrows are equivalences. Using the universal property of $\Pi^{\NN,\et}_\stX$ one can conclude that $\alpha$ fits in the commutative diagram. Thus $\alpha^*V\simeq b^*(a_*V)$, that is there exists an $F$-divided sheaf on $\Gamma$ pulling back to $\E$ along $\stX\arr \Gamma$.
  Take any finite separable extension $L/k$ with a map $\Spec L\arr \Gamma$. As by \ref{voucher for finite stacks},  $\FDiv(L/k)\cong\Vect(L)$, $\sE$ becomes trivial on $\sX\times_{\Gamma} \Spec(L)$. But the map $\sX\times_{\Gamma}\Spec(L)\arr \sX$, as a pullback of $\Spec(L)\arr\Gamma$, is a finite \'etale cover.

The case of $\FDiv_\infty$ is analogous, just replace $\Pi^{\NN,\et}_\stX$ with $\Pi^\NN_\stX$.
\end{proof}

We now concentrate on the converse problem in Theorem \ref{thmI}.
\begin{proof}[Proof of Theorem \ref{thmI}, second part]
 By \ref{thm for Fdiv} it is enough to consider the case of $\Fdiv$.
 The category $\Fdiv(\sX/k)_f$ is a sub Tannakian category of $\Fdiv(\sX/k)$ because in a Tannakian category any subobject or quotient of a trivial object is trivial, $\Fdiv(\sU/k)=\prod_{i\in I}\Fdiv(\sU_i/k)$, where each $\sU_i$ is a connected component of $\sU$ and, for all $i$, $\Fdiv(\sU_i/k)$ is Tannakian over $\Hl^0(\sO_{\sU_{i}})_{\eet}$ by \ref{thm for Fdiv}.
 
  For all finite extensions of fields $l/k$ we have the following 2-commutative diagram.
  \[
  \begin{tikzpicture}[xscale=4.0,yscale=-1.4]
    \node (A0_0) at (0, 0) {$\Pi_{\Fdiv(\sX\times_kl/l)}$};
    \node (A0_1) at (1, 0) {$\Pi_{\Fdiv(\sX)}\times_kl$};
    \node (A1_0) at (0, 1) {$\Pi_{\Fdiv(\sX\times_kl/l)_{f\times_k l}}$};
    \node (A1_1) at (1, 1) {$\Pi_{\Fdiv(\sX)_f}\times_kl$};
    \path (A0_0) edge [->]node [auto] {$\scriptstyle{c}$} (A0_1);
    \path (A1_0) edge [->]node [auto] {$\scriptstyle{a}$} (A1_1);
    \path (A0_1) edge [->>]node [auto] {$\scriptstyle{b}$} (A1_1);
    \path (A0_0) edge [->>]node [auto] {$\scriptstyle{d}$} (A1_0);
  \end{tikzpicture}
  \]
 By \ref{base change for Fdiv}, $c$ is an isomorphism. So, to prove that $\Pi_{\Fdiv(\sX/k)_f}$ is finite and \'etale, we may replace $k$ by a finite field extension. In particular we can assume that all connected components of $\stU$, $\stU\times_\stX \stU$, $\stU\times_\stX\stU\times_\stX\stU$ are geometrically connected or, using \ref{thm for Fdiv}, that the rings $\Hl^0(\odi{\stU})_\et$, $\Hl^0(\odi{\stU\times_\stX\stU})_\et$ and $\Hl^0(\odi{\stU\times_\stX\stU\times_\stX\stU})_\et$ are product of copies of $k$. In all the cases considered we can also assume that $\stU$ is geometrically connected.

Since $\Fdiv$ is an fpqc stack, $\Fdiv(\sX/k)_f$ is equivalent to the category of trivial objects $\sO_{\sU}^{\oplus n}\in\Fdiv(\sU/k)$ equipped with descent data, i.e. an isomorphism between the two pull-backs of the trivial object along $\sU\times_{X}\sU\rightrightarrows\sU$ satisfying the cocycle condition.

Denotes by $I,J$ the set of connected components of $\stU\times_\stX\stU$ and $\stU\times_\stX\stU\times_\stX\stU$ respectively.
Restricting to each connected components and using \ref{thm for Fdiv}, we see that the pullback functors 
$$\Vect(\Spec(\Hl^0(\sO_{\sU})_{\text{\'et}})\simeq\Fdiv(\Spec(\Hl^0(\sO_{\sU})_{\text{\'et}})/k)\arr \Fdiv(\sU/k)$$  
$$\Vect(\Spec(\Hl^0(\sO_{\sU\times_{\sX}\sU})_{\text{\'et}}))\simeq\Fdiv(\Spec(\Hl^0(\sO_{\sU\times_{\sX}\sU})_{\text{\'et}})/k)\arr \Fdiv((\sU\times_{\sX}\sU)/k)$$ 
$$\Vect(\Spec(\Hl^0(\sO_{\sU\times_{\sX}\sU\times_{\sX}\sU})_{\text{\'et}}))\simeq\Fdiv(\Spec(\Hl^0(\sO_{\sU\times_{\sX}\sU\times_{\sX}\sU})_\et)/k)\arr \Fdiv(\sU\times_{\sX}\sU\times_{\sX}\sU/k)$$
are fully faithful. Thus $\FDiv(\stX/k)_f$ is equivalent to the category of vector bundles $V$ on $\Spec(\Hl^0(\sO_{\sU})_{\et})=\Spec k$ with an isomorphism between the two pullbacks of $V$ to $\Spec(\Hl^0(\sO_{\sU\times_{\sX}\sU})_{\eet})=\Spec k^I$ satisfying the cocycle condition in $\Spec(\Hl^0(\sO_{\sU\times_{\sX}\sU\times_{\sX}\sU})_\et)=\Spec k^J$.

So $V$ is a $k$-vector space and the data of the isomorphism between the pullbacks is a collection of automorphisms $\sigma_Z\colon V\arr V$ for all $Z\in I$. Denote by $\pr_{ij}\colon \stU\times_\stX\stU\times_\stX\stU\arr \stU\times_\stX\stU$ and $\pr_i\colon \stU\times_\stX\stU \arr \stU$ the projections. Given $D\in J$ we denote by $Z_{ij,D}\in I$ the unique connected component containing $\pr_{ij}(D)$. It is easy to see that the cocycle conditions on the collection $(\sigma_Z)_{Z\in I}$ translates into the relation
\[
\sigma_{Z_{13,D}}=\sigma_{Z_{23,D}}\circ \sigma_{Z_{12,D}} \text{ for all }D\in J
\]
If $G$ is the quotient of the free non-abelian group over the symbols $(e_Z)_{Z\in I}$ by the relations $e_{Z_{13,D}}^{-1}e_{Z_{23,D}}e_{Z_{12,D}}$ for $D\in J$, it follows that $\FDiv(\stX/k)_f\simeq \Rep_k G$.
If $f$ is geometrically connected then $|I|=|J|=1$ and therefore $G$ is trivial as required. So assume that either $f$ is proper or $\stU$ is geometrically irreducible.
We are going to show that $|G|\leq |I|$.

Consider the Cartesian diagram
$$\xymatrix{\sU\times_\sX\sU\times_\sX\sU\ar[r]^-{\pr_{23}}\ar[d]^-{\pr_{12}}&\sU\times_\sX\sU\ar[d]^-{\pr_1}\\\sU\times_\sX\sU\ar[r]^-{\pr_2}& \sU}$$
Given $Z,W\in I$ we have $\pr_1(Z)\cap\pr_2(W)\neq \emptyset$: if $f$ is proper then $\pr_1(Z)$ and $\pr_2(W)$ are open and closed and thus equal to $\stU$, if $\stU$ is irreducible it is the intersection of two non-empty open subsets. In particular there exists $D\in J$ such that $Z_{12,D}=W$, $Z_{23,D}=Z$ and therefore $e_Ze_W=e_{Z_{13,D}}$ in $G$. Thus it remains to show that for all $Z\in I$ there exists $W$ such that $e_Z^{-1}=e_W$ in $G$.

Notice that the diagonal in $J$ determines the relation $e_\Delta=e_\Delta^2$, where $\Delta\in I$ is the diagonal. Thus $e_\Delta=1$ in $G$. Now consider the Cartesian diagram
$$\xymatrix{\sU\times_\sX\sU\times_\sX\sU\ar[r]^-{\pr_{13}}\ar[d]^-{\pr_{12}}&\sU\times_\sX\sU\ar[d]^-{\pr_1}\\\sU\times_\sX\sU\ar[r]^-{\pr_1}& \sU}$$
Since $\pr_1(Z)\subseteq \pr_1(\Delta)=\stU$, there exists $D\in J$ such that $Z_{12,D}=Z$ and $Z_{13,D}=\Delta$. Thus $e_{Z_{23,D}}e_Z=e_\Delta=1$ as required.
\end{proof}

\begin{rmk}
In the above proof the crucial property of $\Fdiv$ we used is that it is a stack in the fppf topology. Whether this property is true also for $\Strat$ and $\Crys$ is unclear.
\end{rmk}




\begin{lem}\label{ess finite descends along Frobenius}
Let $\stX$ be an inflexible and pseudo-proper fiber category over $k$, $V\in \Vect(\stX)$ and denote by $F\colon \stX\arr\stX$ the absolute Frobenius. If there exists $m\in \N$ such that $F^{m*}V\in \Vect(\stX)$ is essentially finite then $V$ is essentially finite too.
\end{lem}
\begin{proof}
 We can consider the case $m=1$ only. Set $n=\rk V$. The vector bundle $V$ is given by an  $\F_p$-map $v\colon \stX\arr \sB_{\F_p}(\GL_{n,\F_p})$: the stack $\sB_{\F_p}(\GL_{n,\F_p})$ has a universal vector bundle $\E$ of rank $n$ such that $v^*\E\simeq V$. By  \ref{monodromy essentially finite} $V$ is essentially finite if and only if $v$ factors as $\stX\arrdi\phi \Gamma \arr \sB_{\F_p}(\GL_{n,\F_p})$ where $\Gamma$ is a finite $k$-gerbe and $\phi$ is $k$-linear. The vector bundle $F^*V$ corresponds to the composition $\stX\arrdi v \sB_{\F_p}(\GL_{n,\F_p})\arrdi{\overline F} \sB_{\F_p}(\GL_{n,\F_p})$, where $\overline F$ is the absolute Frobenius of $\sB_{\F_p}(\GL_{n,\F_p})$. Thus we have a diagram
   \[
  \begin{tikzpicture}[xscale=2.0,yscale=-1.2]
    \node (A0_0) at (0, 0) {$\stX$};
    \node (A0_1) at (1, 0) {$\Delta$};
    \node (A0_2) at (2, 0) {$\sB_{\F_p}(\GL_{n,\F_p})$};
    \node (A1_1) at (1, 1) {$\Gamma$};
    \node (A1_2) at (2, 1) {$\sB_{\F_p}(\GL_{n,\F_p})$};
    \path (A0_0) edge [->]node [auto] {$\scriptstyle{}$} (A0_1);
    \path (A0_0) edge [swap,->]node [auto] {$\scriptstyle{\phi}$} (A1_1);
    \path (A0_1) edge [->]node [auto] {$\scriptstyle{}$} (A1_1);
    \path (A0_2) edge [->]node [auto] {$\scriptstyle{\overline F}$} (A1_2);
    \path (A1_1) edge [->]node [auto] {$\scriptstyle{}$} (A1_2);
    \path (A0_0) edge [->,bend right=40]node [auto] {$\scriptstyle{v}$} (A0_2);
    \path (A0_1) edge [->]node [auto] {$\scriptstyle{}$} (A0_2);
  \end{tikzpicture}
  \]
where $\Gamma$ is a finite $k$-gerbe and the square is $2$-Cartesian. We conclude by showing that $\Delta$ is a finite gerbe over $k$.

The map $\overline F\colon \sB_{\F_p}(\GL_{n,\F_p})\arr \sB_{\F_p}(\GL_{n,\F_p})$ is induced by the Frobenius of $\GL_{n,\F_p}$. Since this last map is a surjective group homomorphism with finite kernel it follows that $\overline F$ and therefore $\Delta\arr \Gamma$ is a finite relative gerbe. This plus the assumption that $\Gamma$ is a finite gerbe implies that $\Delta$ is a finite gerbe.
\end{proof}

\begin{proof}[Proof of Corollary \ref{corI}]
 If $V$ is an essentially finite vector bundle on $\stX$ then, by \ref{monodromy essentially finite}, there exist $\psi\colon \stX\arr \Gamma$, where $\Gamma$ is a finite gerbe and $W\in \Vect(\Gamma)$ such that $\psi^*W\simeq V$. If $\Spec L\arr \Gamma$ is any map from a finite field extension of $k$ then $f\colon \stU=\stX\times_\Gamma \Spec L\arr \stX$ is a finite and flat map and the pullback along this map of $V$ is free. Since $\stU\times_\stX\stU\arr\stX$ is finite and flat, the pushforward of the structure sheaf is locally free on $\stX$ and therefore the set of its global sections form a finite dimensional $k$-vector space.
 
 Consider now a proper, flat and surjective map $f\colon \stU\arr \stX$ as in the statement and $V\in \Vect(\stX)$ such that $f^*V$ is free. We are first going to extend $V$ to some object in $\FDiv_i(\stX/k)$ trivialized by $f$ for some $i>0$.
 If $\stZ$ is a stack of finite type over $k$ and we apply \ref{thm for Fdiv} on the connected components we see that $\FDiv(\stZ/k)\arr \Vect(\stZ)$ is faithful and the set of endomorphisms of the unit object of $\FDiv(\stZ/k)$ is identified with $\Hl^0(\odi\stZ)_\et$. In particular for all $i$ we have 
 \[
\End_{\FDiv_i(\stZ/k)}(\mathds{1})=\{x\in \Hl^0(\odi\stZ)\st x^{p^i}\in \Hl^0(\odi\stZ)_\et\}\subseteq \Hl^0(\odi\stZ)
\]
via the functor $\FDiv_i(\stZ/k)\arr \Vect(\stZ)$. In particular if $\Hl^0(\odi\stZ)$ is a finite $k$-algebra the above inclusion is an equality for $i$ big enough, which also means that the functor $\FDiv_i(\stZ/k)\arr \Vect(\stZ)$ restricted to the full subcategory of free objects is fully faithful. 

By fppf descent along $f\colon\stU\arr \stX$ the vector bundle $V$ is given by a free object of rank $\rk V=r$ on $\stU$ with an isomorphism of the two pullbacks in $\stU\times_\stX\stU$ satisfying the cocycle condition of the triple product. By discussion above for $i$ big enough this also determines a descent data on the free object of rank $r$ in $\FDiv_i(\stU/k)$. Since $\FDiv_i$ and $\Vect$ are stacks in the fppf topology there exists $E\in \FDiv_i(\stX/k)$ of the form  $(V,W,\lambda)$ with $W\in \FDiv(\stX/k)$  such that $f^*E\in \FDiv_i(\stU/k)$ is trivial. By the definition of $\FDiv_i(\stX/k)$, $W_0\simeq F^{i*}V$, where $F$ is the absolute Frobenius of $\stX$, and $f^*W$ is trivial in $\FDiv(\stU/k)$. By \ref{ess finite descends along Frobenius} and Theorem \ref{thmI} we can conclude that $V$ is essentially finite.
\end{proof}

\begin{lem}\label{additive torsors}
 Let $\stX$ be an algebraic stack over a field $k$ of positive characteristic such that $\Hl^1(\sX,\odi\stX)$ is a finite dimensional $k$-vector space. Then any $\G_a^r$-torsor over $\stX$ comes from a torsor under a finite subgroup scheme of $\G_a^r$
\end{lem}
\begin{proof}
 We can assume $r=1$. We interpret $\Hl^1(\odi\stX)$ as the set of isomorphism classes of $\G_a$-torsors over $\stX$. If $\phi\colon \G_a\arr \G_a$ is a group homomorphism then it induces a map of sets $\Lambda(\phi)\colon \Hl^1(\odi\stX)\arr \Hl^1(\odi\stX)$. Thus we obtain a map 
 \[
 \Lambda\colon \End_{\text{groups}}(\G_a) \arr \End_{\text{\sets}}(\Hl^1(\odi\stX))
 \]
 Both sides are left $k$-algebras and, by the definition of the $k$-vector space structure on $\Hl^1(\odi\stX)$, $\Lambda$ is a morphism of $k$-algebras.
 
 Consider the relative Frobenius $F\colon \G_a\arr \G_a$, that is $F(x)=x^p$ functorially and $R\in k[y]$. The map $R(F)\colon \G_a\arr \G_a$ is a group homomorphism and, since $\G_a$ is connected and reduced, $R(F)$ is also surjective unless $R=0$ and in this case its kernel is finite. If $R(y)=\sum_j \lambda_jy^j$ and $v\in \Hl^1(\odi\stX)$ then
 \[
 \Lambda(R(F))v=\sum_j \lambda_j \Lambda(F)^j (v)
 \]
 Since $\Hl^1(\odi\stX)$ is a finite dimensional $k$-vector space, the vectors $v$, $\Lambda(F)(v)$, $\Lambda(F)^2(v)$ and so on must be eventually linearly dependent. Thus there exists a non zero polynomial $R\in k[y]$ such that $\Lambda(R(F))v=0$. This means that the $\G_a$-torsor $v$ become trivial under the map $\sB_k\G_a\arr\sB_k\G_a$ induced by $R(F)$. Since $R(F)$ is surjective, it follows that $v$ comes from a torsor under the finite group scheme $\Ker(R(F))$.
\end{proof}

\begin{lem}\label{extensions of O}
 Let $\stX$ be an algebraic stack over a field $k$ of positive characteristic such that $\dim_k \Hl^1(\sX,E) <\infty$ for all vector bundles $E$ on $\stX$. Let 
 \[
 \shG_0 \arr \shG_1 \arr \cdots \arr \shG_{N-1} \arr \shG_N=0
 \]
 be a sequence of surjective maps of quasi-coherent sheaves on $\stX$ such that $\Ker(\shG_{l-1} \arr \shG_l)$ is free of finite rank for $1\leq l \leq N$. Then there exists a finite flat surjective map $f\colon \stX'\arr \stX$ such that $f^*\shG_l$ is free of finite rank for all $l$.
\end{lem}
\begin{proof}
 Using induction it is enough to show that if $0\arr \odi\stX^m \arr \shG\arr \odi\stX^n\arr 0$ is an exact sequence there exists a finite flat surjective map $f\colon \stX'\arr \stX$ such that $f^*\shG$ is free. The sheaf $\shG$ is given by an element $x\in \Ext^1(\odi\stX^n,\odi\stX^m)$. Moreover there exists an isomorphism $\Ext^1(\odi\stX^n,\odi\stX^m)\simeq \Hl^1(\odi\stX)^{mn}$ functorial in $\stX$, so that $x$ corresponds to a sequence $x_i\in \Hl^1(\odi\stX)$. If $x_i$ corresponds to the $\G_a$-torsor $h_i\colon \stP_i\arr \stX$, then by construction $h_i^*x_i=0$. By \ref{additive torsors} there exists a finite flat surjective map $f_i\colon \stX_i\arr\stX$ factoring through $h_i\colon \stP_i\arr\stX$. In particular $f_i^*x_i=0$. Thus $f=\prod_i f_i \colon \prod_i \stX_i \arr \stX$ is a finite flat surjective map such that $f^*x=0$, which means that $f^*\shG\simeq f^*\odi\stX^m\oplus f^*\odi\stX^n$.
\end{proof}

\begin{proof}[Proof of Corollary \ref{corII}]
 By \ref{Tannakian virtually unipotent} we must prove that if $E$ is a virtually unipotent vector bundle on $\stX$ then it is essentially finite. Thanks to \ref{ess finite descends along Frobenius} we can assume that $E$ is an extended essentially finite sheaf. We are going to show that there exists a finite and flat map $f\colon \stX'\arr \stX$ such that $f^*E$ is free. Since $\stX$ is inflexible the conclusion will follow from Corollary \ref{corI}.
 
 Since $\stX$ is inflexible there exists a map $\stX\arr \Gamma$ to a finite gerbe such that all essentially finite quotients $E_i/E_{i+1}$ of a filtration of $E$ as in \ref{filtration vu} comes from $\Gamma$. If $L/k$ is a finite extension with $\Gamma(L)\neq \emptyset$, the pullback of $\Spec L\arr \Gamma$ along $\stX\arr \Gamma$ gives a finite flat surjective  map $g\colon \stY\arr \stX$ trivializing all essentially finite quotients $E_i/E_{i+1}$. Applying \ref{extensions of O} on the sheaf $(g^*E)^\vee$ over $\stY$ we find a finite flat surjective  map $h\colon \stX'\arr \stY$ such that $h^* g^* E ^\vee$ is free. The composition $f\colon \stX'\arrdi h\stY\arrdi g \stX$ gives the desired finite flat surjective map. 
\end{proof}

\section{Trivializations by dominant maps}

We fix a field $k$ of characteristic $p>0$. For the definition and properties of geometrically unibranch stacks we refer to Appendix \ref{Unibranch}.
In this section we are going to prove Theorem \ref{thmII} and deduce Corollary \ref{corIII} from it. The crucial point is the following result.

\begin{thm} \label{dense open for unibranch}
Let $\sX$ be a geometrically unibranch algebraic stack locally of finite type over  $k$ and $\sU\subseteq \sX$ be a dense open subset. Then the functor $\FDiv(\stX/k)\arr \FDiv(\stU/k)$ is fully faithful and stable under sub-objects, that is if $F\in \FDiv(\stX)/k$ and $E_\stU\subseteq F_{|\stU}$ in $\FDiv(\stU/k)$ then there exists $E\subseteq F$ in $\FDiv(\stX/k)$ inducing the given inclusion.

In particular if $\stX$ is geometrically irreducible and quasi-compact then $\Pi_{\FDiv(\stU/k)}\arr \Pi_{\Fdiv(\stX)/k}$ is a quotient of $k$-gerbes.
\end{thm}

We first show how to obtain Theorem \ref{thmII} from the above result.

\begin{proof}[Proof of Theorem \ref{thmII} using Theorem \ref{dense open for unibranch}]
We have to prove the ``if'' part and, thanks to \ref{thm for Fdiv}, we just have to consider the case of $\FDiv$.
 Let $f\colon \stY\arr \stX$ be a dominant map and $E\in \FDiv(\stX/k)$ which is trivialized by $f$. We want to show that $E$ is essentially finite.  Replacing $\sY$ by an atlas we can assume that $\stY=Y$ is a scheme. 
 
 First we assume that $f$ is also flat. In this case, we take a non-empty irreducible open subset $\bar{Y}_0\subseteq Y\times_k\bar{k}$ and let $Y_0\subseteq Y\times_k{k'}$ be a model of $\bar{Y}_0$ over a finite field extension $k'/k$. By \ref{geometrically unibranch on geometric fiber}, \ref{base change for Fdiv} and \ref{essentially finite and base change} we can assume $k'=k$. Consider the image $\sU:=f(Y_0)\subseteq \sX$. From Theorem \ref{thmI}, 3) we see that $E|_{\sU}$ is an essentially finite object over $\sU/k$ and, from  \ref{dense open for unibranch}, we can conclude that also $E$ is essentially finite in $\Fdiv(\sX/k)$ as required.
 
 Now we come back to the general case. Let $X\arr \sX$ be a smooth atlas with $X$ quasi-compact. Using \ref{geometrically unibranch on geometric fiber}, \ref{base change for Fdiv} and \ref{essentially finite and base change} we may extend $k$ a little bit and assume that $X$ has a $k$-rational point $x$. Using the first part of Theorem   \ref{thmI} and the flat case we may replace $\sX$ by the connected component of $x$ in $X$ and assume that $\sX=X$ is a scheme. Using \ref{insensible to nilpotent thickenning} we may assume that $X$ is a reduced scheme. Since $X$ is also geometrically unibranch and connected, it is integral. By \cite[Theorem 24.3]{Mat} the locus $V\subseteq Y$ of points flat over $X$ is open and, since $f$ is dominant, it is non-empty. Replacing $Y$ by $V$ we are in the case when $f$ is flat.
 \end{proof}

\begin{ex} \label{not unibranch}
In Theorem \ref{dense open for unibranch} the hypothesis that $\sX$ is geometrically unibranch is necessary. Assume that $k$ is an algebraically closed field of odd characteristic and 
consider $\sX=\Spec(A)$ where $A=k[t]_t[x,y]/(x^2-y^2t)$, which is an affine integral scheme of finite type over $k$, and its open subset $\stU=\Spec(A_y)$. We claim that $\Pi_{\FDiv(\stU/k)}\arr \Pi_{\FDiv(\stX/k)}$ is not a quotient. Since the profinite quotient of $\FDiv$ is $\Pi^{\NN,\et}$ and, choosing a rational point $u\in \stU$, this corresponds to the Grothendieck \'etale fundamental group $\pi^\et$, it is enough to show that $\pi^\et(\stU,u)\arr \pi^\et(\stX,u)$ is not surjective. 
The scheme $P=\Spec(A[T]/(T^2-t))$ is a $\mu_2$-torsor over $\stX$ and it is non-trivial since $t$ is not a square in $A$. 
On the other hand $P$ becomes trivial when restricted to $\stU$, which implies that there is a quotient $\pi^\et(\stX,u)\arr \mu_2$ such that the composition $\pi^\et(\stU,u)\arr \pi^\et(\stX,u)\arr \mu_2$ is trivial.
\end{ex}

Before proving Theorem \ref{dense open for unibranch} we collect some preparatory results.


\begin{prop} \label{completion is trivial}
Assume that $k$ is a field whose absolute Frobenius is finite and
let $R$ be a complete local $k$-algebra whose residue field is finite over $k$. Then
\[
\FDiv(R/k)\simeq \Vect(R_\et)
\]
\end{prop}
\begin{proof}
By \cite[Theorem 6.17]{TZ} it follows that $\FDiv(R/k)$ is an $R_\et$-Tannakian category. Thus it is enough to show that any object $ E\in \FDiv(R/k)$ is trivial. Let $\mathfrak{m}$ be the maximal ideal of $R$, $L$ its residue field and $F\in \FDiv(R/k)$ be the free object of rank $\rk E$. Using that $\FDiv$ is insensible to nilpotent thickenings (\ref{insensible to nilpotent thickenning}) and that $\FDiv(L/k)\simeq \Vect(L_\et)$ (\ref{voucher for finite stacks}), there is a compatible system of isomorphisms $\sigma_n\colon E\otimes_R R/\mathfrak{m}^n\arr F\otimes_R R/\mathfrak{m}^n$ in $\FDiv((R/\mathfrak{m}^n)/k)$. Since the absolute Frobenius of $k$ is finite all rings $R^{(i)}$ are complete with respect to $\mathfrak{m}^{(i)}$. As $E_i$ are free $R^{(i)}$-modules they are complete with respect to $\mathfrak{m}^{(i)}$. Thus these $\sigma_n$ extend to an isomorphism $\sigma\colon E\arr F$ as required.
\end{proof}

\begin{lem}\label{intersection}
Let $A$ be a ring, $A\arr A'$ be a faithfully flat map and $ A\arr B$ be an injective map. Then $A= A'\bigcap B$ inside $A'\otimes_AB$.
\end{lem}
\begin{proof} 
 We provide two proofs:\\
(1) Consider the map $A\xrightarrow{\phi} A'\bigcap B\subseteq B$. It is enough to show that $\phi':=\phi\otimes_AA'$ is an isomorphism. For this it is enough to show that for any element $x\in  A'\bigcap B$, $x\otimes 1\in B\otimes_AA'$ is in the image of $A'\arr  B\otimes_AA'$. But by the definition of $A'\cap B$, there exists $a'\in A'$ such that $1\otimes a'= x\otimes 1\in B\otimes_AA'$. This means that $\phi'$ is an isomorphism, whence the claim.

\noindent (2) Consider the diagram $$\xymatrix{A\ar[r]\ar@{^{(}->}[d]&A'\ar@{^{(}->}[d]\ar@/^/[r]^-{f_1}\ar@/_/[r]_-{f_2}&A'\otimes_AA'\ar@{^{(}->}[d]\\B\ar[r]&B\otimes_AA'\ar@/^/[r]^-{g_1}\ar@/_/[r]_-{g_2}&B\otimes_AA'\otimes_AA'}$$ with canonical maps. From the diagram it is clear that $A'\bigcap B$, as a subset of $A'$, is contained in the subset of equalizers of $f_1$ and $f_2$, i.e. $A'\bigcap B\subseteq A$. Thus $A'\bigcap B= A$.  
\end{proof}

\begin{lem}\label{endomorphism of the trivial object}
Let $A$ be a Noetherian complete local domain over a perfect field $k$ with fraction field $K$ and assume that its residue field is finitely generated over $k$. Then $\FDiv(K/k)$ is $K_\et$-Tannakian category and $K_\et/k$ is finite.
\end{lem}
\begin{proof}
By \ref{thm for Fdiv} and \ref{cases for thm for Fdiv} we can conclude that $\FDiv(A/k)$ is an $A_\et$-Tannakian category and $\FDiv(K/k)$ is an $F$-Tannakian category where $F=\End_{\FDiv(K/k)}(\mathds{1})$.
Let $B$ be the integral closure of $A$ in $K$. Since $A$ is Nagata \cite[\href{http://stacks.math.columbia.edu/tag/032E}{032E}, Lemma 10.156.2 and 10.156.8]{SP}), $B$ is finite over $A$ and since $A$ is Henselian, $B$ is a local domain.  Thus we can assume that $A$ is normal and in this case we will show that $F\subseteq A_\et$. This will end the proof because $K_\et\subseteq F \subseteq A_\et\subseteq K_\et$ and because $A_\et\subseteq L_\et$ is finite, where $L$ is the residue field which is finitely generated over $k$.

Pick $x\in F$, that is $x=(x_n)_{n\in \N}$ with $x_n\in K$ and $x_{n+1}^p=x_n$ (here we use that $K^{(i)}\simeq K$ because $k$ is perfect). Given a discrete valuation $v\colon K\arr \Z$ one can easily conclude that $v(x_n)=0$ for all $n\in \N$, which means that $x_n\in A$ for all $n\in \N$. In particular $x\in \End_{\FDiv(A/k)}(\mathds{1})=A_\et$ as desired.
\end{proof}

\begin{lem}\label{twist to get geo irreducible}
 Let $\stX$ be a stack of finite type over $k$. Then there exists a finite and purely inseparable extension $L/k$ such that $(\stX\times_k L)_\red$ is geometrically reduced.
\end{lem}
\begin{proof}
 Taking a smooth covering by an affine scheme one can assume $\stX=\Spec A$. Consider the exact sequence $$0\arr I \arr A\otimes_kk^{\frac{1}{p^{\infty}}}\arr (A\otimes_kk^{\frac{1}{p^{\infty}}})_\red\arr 0 $$
 where $k^{\frac{1}{p^{\infty}}}$ be the perfect closure of $k$ and   $I\subseteq A\otimes_k k^{\frac{1}{p^{\infty}}}$ is the ideal of nilpotents. Since $I$ is finitely generated, it is defined over a finite intermediate extension $k\subseteq L\subseteq k^{\frac{1}{p^{\infty}}}$. Extending $A$ to $L$ we may assume that $L=k$, i.e. $\exists I_k\subseteq A$ such that $I_k\otimes_kk^{\frac{1}{p^{\infty}}}=I$. Since $I_k\subseteq I$ is nilpotent,   $A_\red=A/I_k$ is geometrically reduced. 
\end{proof}

\begin{proof}[Proof of Theorem \ref{dense open for unibranch}]
We start by showing the fully faithfulness. The restriction functor $\FDiv(\sX/k)\arr\FDiv(\sU/k)$ is faithful even when $\sX$ is not geometrically unibranch. To see this we may assume that $\sX$ is connected. Let $f:E\arr F\in \FDiv(\sX/k)$ whose restriction to $\sU$ is 0. Then the image $H$ of $f$ is an object in $\FDiv(\sX/k)$ whose restriction to $\sU$ is 0. As the rank of $H$ is constant and $\sU$ is non-empty, $H=0$, i.e. $f|_\sU=0$.

Since $\FDiv$ is a stack in the fpqc topology, using the faithfulness one can easily reduce to the case $\stX=X=\Spec A$ is affine. Using \ref{twist to get geo irreducible} and that $\FDiv(X/k)\simeq\FDiv(X^{(i)}/k)\simeq \FDiv((X^{(i)})_\red/k)$ we can further assume that $X$ and $U=\stU$ are also geometrically integral by extending $X$ to a finite   extension of $k$, restricting to one connected component, and then taking reduction. We can also assume that $U=\Spec A_a$ with $a\in A$. 

Let  $E,F\in\Fdiv(X/k)$ and consider $\phi_U\in\Hom_{\Fdiv(U)}(E|_U,F|_{U})$. Since all $X^{(i)}$ are integral it is enough to extend  each map  $E_i|_U\arr F_i|_U$ to $E_i\arr F_i$ and, up to replace $X$ by $X^{(i)}$, we just have to show this extension for $i=0$. Shrinking $X$ we can assume that $E_0=A^{\oplus n}$ and $F_0=A^{\oplus m}$, so that $\phi_U$ is a matrix with coefficients in $A_a$ and we must prove those coefficients belong to $A$. Using the functor $\Fdiv(-/k)\arr \FDiv(-/\overline k)$ and \ref{intersection} we can assume that $k$ is algebraically closed. Denote by $K$ the fraction field of $A$ and by $\phi_K \colon E\otimes K\arr F\otimes K$ the restriction of $\phi_U$. Consider $\mathfrak{p}\in X\setminus U$ a closed point and the diagram $$\xymatrix{A_\mathfrak{p}^{\oplus n}
\ar@{.>}[r]\ar@{^{(}->}[d] & A_\mathfrak{p}^{\oplus m}\ar@{^{(}->}[d]\\K^{\oplus n}
\ar[r]^-{(\phi_K)_0} & K^{\oplus m}}$$
We must show the existence of the dashed arrow, or in other words, that the coefficients of the matrix $\phi_K$ are in $A_\mathfrak{p}$. Using again \ref{intersection}, it is enough to show that $\FDiv(\hat A_\mathfrak{p}/k)\arr \FDiv(\hat K/k)$ is fully faithful, where $\hat A_\mathfrak{p}$ and $\hat K$ are the completion of $A_\mathfrak{p}$, which is a domain because $A$ is geometrically unibranch (\ref{basic properties}), and its fraction field respectively. This follows because $\FDiv(\hat A_\mathfrak{p}/k)=\Vect (k)$ by \ref{completion is trivial} and $\FDiv(\hat K/k)$ is a $k$-Tannakian category by \ref{endomorphism of the trivial object}.

We now show that the full subcategory $\Fdiv(\sX/k)\subseteq \Fdiv(\sU/k)$ is stable under taking subobjects.  Let $F\in \Fdiv(\sX/k)$, and let $E_{\sU}\subseteq F|_{\sU}$. We will show that there exists $E\subseteq F$ such that $E|_\sU=E_\sU$. By the fully faithfulness which we just proved, if the extension $E$ exists it is unique. Since $\FDiv$ is a stack in the \'etale topology we can assume $\stX=X=\Spec A$ and also $\sU=U=\Spec A_a$ for some $a\in A$. As above we can further assume that $X$ and $U$ are geometrically integral. Denote by $j\colon U\arr X$ the open embedding. We have an infinite sequence of Cartesian diagrams 
  \[
  \begin{tikzpicture}[xscale=2.0,yscale=-1.2]
    \node (A0_0) at (0, 0) {$U$};
    \node (A0_1) at (1, 0) {$U^{(1)}$};
    \node (A0_2) at (2, 0) {$U^{(2)}$};
    \node (A0_3) at (3, 0) {$\cdots$};
    \node (A1_0) at (0, 1) {$X$};
    \node (A1_1) at (1, 1) {$X^{(1)}$};
    \node (A1_2) at (2, 1) {$X^{(2)}$};
    \node (A1_3) at (3, 1) {$\cdots$};
    \path (A0_1) edge [->]node [auto] {$\scriptstyle{j^{(1)}}$} (A1_1);
    \path (A0_0) edge [->]node [auto] {$\scriptstyle{}$} (A0_1);
    \path (A0_1) edge [->]node [auto] {$\scriptstyle{}$} (A0_2);
    \path (A1_0) edge [->]node [auto] {$\scriptstyle{}$} (A1_1);
    \path (A1_1) edge [->]node [auto] {$\scriptstyle{}$} (A1_2);
    \path (A0_2) edge [->]node [auto] {$\scriptstyle{j^{(2)}}$} (A1_2);
    \path (A1_2) edge [->]node [auto] {$\scriptstyle{}$} (A1_3);
    \path (A0_2) edge [->]node [auto] {$\scriptstyle{}$} (A0_3);
    \path (A0_0) edge [->]node [auto] {$\scriptstyle{j}$} (A1_0);
  \end{tikzpicture}
  \]
In particular, since the vertical maps are affine, given $H\in\FDiv(U/k)$ there are isomorphisms $\sigma_n \colon (j^{(n+1)}_*H_{n+1})_{|X^{(n)}}\arr j^{(n)}_*((H_{n+1})_{|U^{(n)}})\simeq j^{(n)}_*H_n$, so we can define $j_*H$ as the "quasi-coherent $F$-divided sheaf" $(j^{(n)}_*H_{n},\sigma_n)$. We have $F \subseteq j_*j^{*}F$ and $j_*(E_U)\subseteq j_*j^{*}F$ and define $E_n:=j_*^{(n)}(E_U)_n \bigcap F_n$. By construction  $j^{(n)^*}E_n = (E_U)_n$, and there are canonical maps $\tau_n:(E_{n+1})_{|X^{(n)}}\arr E_n$ which are compatible with the transition isomorphisms of $F$. We are going to prove that those $\tau_n$ are isomorphisms, that is $(E,\tau_n)\in \FDiv(X/k)$. This would conclude the proof, as $(E,\tau_n)$ gives the desired extension.
In order to show that those $\tau_n$ are isomorphisms we can first assume that $k$ is algebraically closed, and by completing at closed points, we may further assume that $A$ is a complete local Noetherian domain with residue field $k$. In particular $\FDiv(X/k)=\Vect (k)$ by \ref{completion is trivial} and $\FDiv(U/k)$ is a $k$-Tannakian category using \ref{thm for Fdiv}. Since a subobject of a trivial object in $\FDiv(U/k)$ is trivial,  $E_U\subseteq F|_U$ is trivial. Thus we can write $F=\tilde E \oplus G$, where $\tilde E,G$ are trivial objects in $\Fdiv(X/k)$, and $\tilde E|_U=E_U$. Now it is clear that $\tilde E_n=j_*^{(n)}(E_U)_n \bigcap F_n=E_n$ and the transition maps must coincide. Thus $\tau_n$ are all isomorphisms.
\end{proof}

\begin{proof}[Proof of Corollary \ref{corIII}]
 Consider a map $\Spec L\arr \Gamma$, where $L/k$ is a finite extension, which exists because $\Gamma$ is of finite type. Since the map $\Spec L\arr \Gamma\times_k L$ is faithfully flat and affine, and  $L$ is smooth over $L$, we can conclude that $\Gamma$ is a smooth stack over $k$ and, in particular, geometrically unibranch and geometrically integral. Since by \ref{voucher for finite stacks} $\FDiv(L/k)=\Vect (L_\et)$, applying Theorem \ref{thmII} to the flat map $\Spec L\arr\Gamma$ we can conclude that all objects of $\FDiv(\Gamma/k)$ and $\FDiv_\infty(\Gamma/k)$ are essentially finite. The remaining claims follow from \ref{thm for Fdiv} taking into account that $\Gamma$ is inflexible over $k$, $\Pi^{\NN,\et}_{\Gamma/k}=\Gamma_\et$ and $\Pi^\NN_{\Gamma/k}=\hat\Gamma$ (see \cite[Definition B8]{TZ}).
\end{proof}

\appendix
\section{Geometrically Unibranch Algebraic Stacks}\label{Unibranch}

\begin{defn}
A local ring $R$ is called geometrically unibranch if, denoted by $R_\red$ its reduction, $R_\red$ is a domain and its integral closure in its fraction field is a local ring whose residue field is a purely inseparable extension of that of $R$.

A scheme $X$ is called geometrically unibranch at a point $x\in X$ if the ring $O_{X,x}$ is geometrically unibranch. A scheme is called geometrically unibranch if it is geometrically unibranch at all its points (see \cite[\href{http://stacks.math.columbia.edu/tag/0BQ1}{0BQ1}]{SP}). 
\end{defn}

\begin{ex}
 All normal schemes are geometrically unibranch.
\end{ex}

\begin{rmk}\label{basic properties}
Let $R$ be a local ring.
 By \cite[pp.100, Definition 2]{Ray} $R$ is geometrically unibranch if and only if the strict Henselization $R^{sh}$ has a unique minimal prime, that is $(R^{sh})_\red$ is a domain. 
 
 By \cite[18.9.1]{EGA4} if $R$ is an excellent local ring and it is geometrically unibranch then the completion $\hat R$ has a unique minimal prime.
 
 If $\Spec R$ is geometrically unibranch then $R$ is geometrically unibranch but the converse is false because this condition does not pass to localizations \cite[Ch 0, 6.5.11, pp. 149]{EGA1}.
\end{rmk}

\begin{rmk}\label{locally irreducible}
 In a geometrically unibranch scheme all irreducible components are also connected components. Indeed the localization in a point lying in an intersection of two different irreducible components has at least two minimal primes. In particular a locally connected (e.g. locally Noetherian) geometrically unibranch scheme is a disjoint union of irreducible and geometrically unibranch schemes. 
\end{rmk}

\begin{lem}\label{geometrically unibranch on geometric fiber}
Let $k$ be a field, $X$ be a $k$-scheme and $\overline x \in X\times_k \overline k$ be a point lying over $x\in X$. Then $X\times_k\overline k$ is geometrically unibranch at $\overline x$ if and only if $X$ is geometrically unibranch at $x$.
\end{lem}
\begin{proof}
We can assume $X=\Spec A$, where $(A,\pp)$ is a local ring and $x=\pp\in \Spec A$.
Let $k^s$ be the separable closure of $k$ and $P\in X\times_k k^s$ lying over $\pp\in X$. For all $k\subseteq L\subseteq k^s$, set $P_L=P\cap (A\otimes_k L)$. Then
\[
\varinjlim_{k\subseteq L\subseteq k^s}(A\otimes_k L)_{P_L}=(A\otimes_kk^s)_P
\]
which implies that $A_\pp$ and $(A\otimes_k k^s)_P$ has the same strict Henselization, so that one is geometrically unibranch if and only if the other is.
In particular we can assume $k=k^s$ separably closed. In this case $A\otimes_k \overline k$ is local so that $\overline x$ corresponds to its maximal ideal. Since $A\arr A\otimes_k \overline k$ is surjective, \textit{purely inseparable} (\cite[Ch I, Prop. 3.7.1, pp. 246]{EGA1}) and integral,  by  \cite[Expos\'e VIII, Th\'eor\` eme 1.1]{SGA4} the pullback along $\Spec(A\otimes_k \overline k)\arr \Spec A$ induces an equivalence of categories between the small \'etale sites of the two schemes.
This easily implies that
\[
(A\otimes_k \overline k)^{sh}\simeq A^{sh}\otimes_A (A\otimes_k \overline k)\simeq A^{sh}\otimes_k \overline k
\]
Thus $\Spec(A\otimes_k \overline k)^{sh}\arr \Spec A^{sh}$ is an homeomorphism and the result is clear.
\end{proof}

\begin{prop}\label{being smooth local}
Let $f\colon X\arr Y$ be a faithfully flat map of schemes, $x\in X$ and $y=f(x)$. If $X$ is geometrically unibranch at $x$ then $Y$ is geometrically unibranch at $y$ and the converse holds if all the geometric fibers of $f$ are normal.
\end{prop}
\begin{proof}
The ``if'' part follows from the fact that $\odi{Y,y}^{sh}\arr \odi{X,x}^{sh}$ is injective and \cite[pp.100, Definition 2]{Ray}. 

For the converse we can assume $Y=\Spec A$, where $(A,\pp)$ is a local and geometrically unibranch domain and $X=\Spec C$. Denote by $P$ the prime ideal of $C$ corresponding to $x\in X$. 
Let $B$ be the integral closure of $A$, which, by assumption, is a local domain whose residue field $l$ is purely inseparable over the residue field $k$ of $A$. 
Given a map $A\arr D$ set $C_D=C\otimes_A D$ and consider the coproducts
  \[
  \begin{tikzpicture}[xscale=1.5,yscale=-1.2]
    \node (A0_0) at (0, 0) {$A$};
    \node (A0_1) at (1, 0) {$C$};
    \node (A0_2) at (2, 0) {$C_P$};
    \node (A0_3) at (3, 0) {$A$};
    \node (A0_4) at (4, 0) {$C$};
    \node (A0_5) at (5, 0) {$C_P$};
    \node (A1_0) at (0, 1) {$B$};
    \node (A1_1) at (1, 1) {$C_B$};
    \node (A1_2) at (2, 1) {$(C_B)_P$};
    \node (A1_3) at (3, 1) {$k$};
    \node (A1_4) at (4, 1) {$C_k$};
    \node (A1_5) at (5, 1) {$(C_k)_P$};
    \node (A2_0) at (0, 2) {$l$};
    \node (A2_1) at (1, 2) {$C_l$};
    \node (A2_2) at (2, 2) {$(C_l)_P$};
    \node (A2_3) at (3, 2) {$l$};
    \node (A2_4) at (4, 2) {$C_l$};
    \node (A2_5) at (5, 2) {$(C_l)_P$};
    \path (A0_2) edge [->]node [auto] {$\scriptstyle{}$} (A1_2);
    \path (A1_1) edge [->]node [auto] {$\scriptstyle{}$} (A2_1);
    \path (A1_0) edge [->]node [auto] {$\scriptstyle{}$} (A2_0);
    \path (A1_5) edge [->]node [auto] {$\scriptstyle{}$} (A2_5);
    \path (A2_1) edge [->]node [auto] {$\scriptstyle{}$} (A2_2);
    \path (A1_2) edge [->]node [auto] {$\scriptstyle{}$} (A2_2);
    \path (A2_3) edge [->]node [auto] {$\scriptstyle{}$} (A2_4);
    \path (A1_4) edge [->]node [auto] {$\scriptstyle{}$} (A1_5);
    \path (A2_4) edge [->]node [auto] {$\scriptstyle{}$} (A2_5);
    \path (A2_0) edge [->]node [auto] {$\scriptstyle{}$} (A2_1);
    \path (A0_3) edge [->]node [auto] {$\scriptstyle{}$} (A1_3);
    \path (A0_4) edge [->]node [auto] {$\scriptstyle{}$} (A0_5);
    \path (A0_4) edge [->]node [auto] {$\scriptstyle{}$} (A1_4);
    \path (A0_1) edge [->]node [auto] {$\scriptstyle{}$} (A1_1);
    \path (A1_4) edge [->]node [auto] {$\scriptstyle{}$} (A2_4);
    \path (A1_0) edge [->]node [auto] {$\scriptstyle{}$} (A1_1);
    \path (A1_1) edge [->]node [auto] {$\scriptstyle{}$} (A1_2);
    \path (A0_0) edge [->]node [auto] {$\scriptstyle{}$} (A1_0);
    \path (A1_3) edge [->]node [auto] {$\scriptstyle{}$} (A1_4);
    \path (A0_0) edge [->]node [auto] {$\scriptstyle{}$} (A0_1);
    \path (A1_3) edge [->]node [auto] {$\scriptstyle{}$} (A2_3);
    \path (A0_3) edge [->]node [auto] {$\scriptstyle{}$} (A0_4);
    \path (A0_5) edge [->]node [auto] {$\scriptstyle{}$} (A1_5);
    \path (A0_1) edge [->]node [auto] {$\scriptstyle{}$} (A0_2);
  \end{tikzpicture}
  \]
Since $B$ is normal and $f$ has normal fibers it follows that $C_B$ and therefore $(C_B)_P$ are normal rings (\cite[Theorem 23.9, pp. 184]{Mat}). Since $A\arr C_P$ is flat and $A\arr B$ is injective and integral, it follows that $C_P\arr (C_B)_P$ is integral and injective. Since $k=A/\pp\arr B/\pp B$ is a purely inseparable morphism, $C_P/\pp C_P\arr (C_B)_P/\pp (C_B)_P$ is purely inseparable too. Thus there is a unique ideal in $(C_B)_P$ which restricts to the maximal ideal of $C_P$. Since a prime ideal in $(C_B)_P$ is maximal if and only if it restricts to a maximal ideal in $C_P$, we conclude that $(C_B)_P$ is a local ring. Thus the connected normal ring $(C_B)_P$ is a domain, and clearly $(C_B)_P$ is the integral closure of $C_P$. The residue field extension of $C_P\arr (C_B)_P$ is purely inseparable because $C_P/\pp C_P\arr (C_B)_P/\pp (C_B)_P$ is a purely inseparable ring morphism.
\end{proof}

\begin{defn}
Let $\stX$ be an algebraic stack and $x\in |\stX|$ be a point. We say that $\stX$ is geometrically unibranch at $x$ if there exists a smooth atlas $U\arr \stX$ and $u\in U$ over $x$ such that $U$ is geometrically unibranch at $u$. The stack $\sX$ is called geometrically unibranch if it is geometrically unibranch at all its points.
\end{defn}
\begin{rmk}
 The notion of geometrically unibranch for algebraic stacks does not depend on the choice of the atlas thanks to \ref{being smooth local}. Moreover it is immediate that an algebraic stack is geometrically unibranch if and only if it has a smooth atlas $U\arr\stX$ with $U$ geometrically unibranch, and one can verify that statements parallel to \ref{geometrically unibranch on geometric fiber} and \ref{being smooth local} remain true for  algebraic stacks.
\end{rmk}

\section{Base Change by Finite Algebraic Extensions}
Let $k$ be a field and $l/k$ be a finite field extension.

\begin{defn}
 Given a $k$-linear category $\shC$ we denote by $\shC\otimes_k l$ the category of pairs $(E,\lambda)$ where $E\in \shC$ and $\lambda\colon l\arr \End_\shC(E)$ is a $k$-algebra map.
\end{defn}
The category $\shC\otimes_k l$ is additive and it has a natural $l$-linear structure. Moreover if $\shD$ is an $l$-linear additive category then a  $k$-linear functor $Q\colon\shD\arr\shC$ extends uniquely to an $l$-linear functor $Q\colon \shD\arr\shC\otimes_k l$.
\begin{rmk}\label{qcoh for finite extension}
 If $\stX$ is a category fibered in groupoids with an fpqc atlas. over $k$ then the pushforward along $\beta\colon \stX\times_k l\arr \stX$ induces an equivalence $\beta_\#\colon  \QCoh(\stX\times_k l)\arr \QCoh(\stX)\otimes_k l$. Because $l/k$ is finite it also restricts to an equivalence $\beta_\#\colon \Vect(\stX\times_k l)\arr \Vect(\stX)\otimes_k l$. Under this equivalence the pullback of $\beta$ sends $\shF\in\QCoh(\stX)$ to $(\shF\otimes_k l,\rho)$ where $\rho$ is induced by the action on the right component of $\shF\otimes_kl$. In particular if $\shC$ is a $k$-Tannakian category then $\shC\otimes_k l$ is an $l$-Tannakian category: $\Pi_\shC\times_k l$ is an $l$-gerbe and $\Vect(\Pi_\shC\times_k l)\simeq \shC\otimes_k l$.
 \end{rmk}

 If $k$ is of characteristic $p>0$, $\stX/k$ is a category fibered in groupoids  with an fpqc atlas, $\stX_l:=\stX\times_k l$ and $\beta\colon\stX_l\arr \stX$ is the projection we have Cartesian diagrams
   \[
  \begin{tikzpicture}[xscale=1.5,yscale=-1.2]
    \node (A1_0) at (0, 1) {$\stX_l$};
    \node (A1_1) at (1, 1) {$\stX_l^{(1,l)}$};
    \node (A1_2) at (2, 1) {$\stX_l^{(2,l)}$};
    \node (A1_3) at (3, 1) {$\cdots$};
    \node (A1_4) at (4, 1) {$\Spec l$};
    \node (A2_0) at (0, 2) {$\stX$};
    \node (A2_1) at (1, 2) {$\stX^{(1)}$};
    \node (A2_2) at (2, 2) {$\stX^{(2)}$};
    \node (A2_3) at (3, 2) {$\cdots$};
    \node (A2_4) at (4, 2) {$\Spec k$};
    \path (A2_3) edge [->]node [auto] {$\scriptstyle{}$} (A2_4);
    \path (A1_4) edge [->]node [auto] {$\scriptstyle{}$} (A2_4);
    \path (A2_1) edge [->]node [auto] {$\scriptstyle{}$} (A2_2);
    \path (A1_0) edge [->]node [auto] {$\scriptstyle{}$} (A1_1);
    \path (A1_3) edge [->]node [auto] {$\scriptstyle{}$} (A2_3);
    \path (A1_1) edge [->]node [auto] {$\scriptstyle{}$} (A1_2);
    \path (A2_2) edge [->]node [auto] {$\scriptstyle{}$} (A2_3);
    \path (A1_3) edge [->]node [auto] {$\scriptstyle{}$} (A1_4);
    \path (A1_0) edge [->]node [auto] {$\scriptstyle{\beta}$} (A2_0);
    \path (A1_1) edge [->]node [auto] {$\scriptstyle{}$} (A2_1);
    \path (A1_2) edge [->]node [auto] {$\scriptstyle{}$} (A1_3);
    \path (A2_0) edge [->]node [auto] {$\scriptstyle{}$} (A2_1);
    \path (A1_2) edge [->]node [auto] {$\scriptstyle{}$} (A2_2);
  \end{tikzpicture}
  \]
This defines a functor $\beta_*\colon \Fdiv(\stX_l/l)\arr \Fdiv(\stX/k)$ which is the right adjoint of the pullback functor  $\beta^*\colon \Fdiv(\stX/k)\arr\Fdiv(\stX_l/l)$.

\begin{prop}\label{base change for Fdiv}
 The unique  $l$-linear functor $\beta_\#\colon \Fdiv(\stX_l/l)\arr \Fdiv(\stX/k)\otimes_k l$ induced by $\beta_*$ is a tensor equivalence. If $\Fdiv(\stX/k)$ is a $k$-Tannakian category then $\FDiv(\stX_l/l)$ is an $l$-Tannakian category (with the obvious tensor structure) and the pullback along $\beta^*\colon \Fdiv(\stX/k)\arr\Fdiv(\stX_l/l)$ corresponds to a $k$-morphism $\Pi_{\Fdiv(\stX_l/l)}\arr \Pi_{\Fdiv(\stX/k)}$ which induces an euqivalence:
 \[
 \Pi_{\Fdiv(\stX_l/l)}\arr \Pi_{\Fdiv(\stX/k)}\times_k l
 \]
\end{prop}
\begin{proof} Consider the following diagram: 
$$\xymatrix{\Fdiv(\sX_l/l)\ar@/^2pc/[rr]_-{\beta_*}\ar[r]^-{\beta_\#}\ar@{.>}[d]^-{\alpha}&\Fdiv(\sX/k)\otimes_kl\ar[r]^-{\textup{For}_{\sX}}\ar[d]^-\cong&\Fdiv(\sX/k)\ar[d]_-a^-\cong\ar@/_3pc/[ll]_-{\beta^*}\\ \Vect(\Pi_{\Fdiv(\sX/k)}\times_kl)\ar@/_2pc/[rr]^-{\gamma_*}\ar[r]^-{\gamma_\#}&\Vect(\Pi_{\Fdiv(\sX/k)})\otimes_kl\ar[r]^-{\textup{For}_{\Pi}}&\Vect(\Pi_{\Fdiv(\sX/k)})\ar@/^3pc/[ll]^-{\gamma^*}}$$
where $\gamma: \Pi_{\FDiv(\stX/k)}\times_kl\arr \Pi_{\FDiv(\stX/k)}$ is the projection, and $\textup{For}_{\Pi}$, $ \textup{For}_{\sX} $ are forgetful functors.
 Applying \ref{qcoh for finite extension} to each $\sX^{(i)}$ it is easy to see that $\beta_\#$ is an equivalence. Assume now that $\Pi_{\FDiv(\stX/k)}$ is a $k$-gerbe. Again by \ref{qcoh for finite extension}   $\gamma_\#$ is an equivalence. Thus we get the dashed arrow $\alpha$, which  makes all the rectangles plus $\gamma_*,\beta_*$ commutative. By the uniqueness of the left adjoint  $\alpha$ also commutes with $\beta^*$, $\gamma^*$ and $a$. Using the fact that $a,\alpha^*,\gamma^*$ are all tensor functors and the fact that for any $M\in\Fdiv(\sX_l/l)$ there exists $N\in \Fdiv(\sX/k)$ with a surjection $\beta^*N\twoheadrightarrow M$ (e.g. $\beta^*\beta_*M\arr M$), one can easily deduce that $\alpha$ is a tensor equivalence. Thus $\Fdiv(\sX_l/l)$ is $l$-Tannakian with the Tannakian gerbe $\Pi_{\FDiv(\stX/k)}\times_kl$. The last claim follows immediately from the 2-commutative diagram.
\end{proof}

\begin{rmk} \label{base change for Fdiv_i} Let $k'/k$ be any field extension and $\sX$ be a geometrically connected algebraic stack of finite type over $k$. Also in this case there is a canonical map of $k'$-gerbes
\[
\delta:\Pi_{\FDiv(\stX\times_k k'/k')}\arr \Pi_{\FDiv(\stX/k)}\times_k k'
\]
and  using the method employed in \cite{De1} one can show that $\delta$ is a quotient if $\stX/k$ is smooth. If $\sX$ is not smooth it is unclear to us whether or not the above morphism is surjective. The main problem is that we don't know whether   the category of quasi-coherent $F$-divided sheaves on $\sX$ is   an abelian category. Thus the techniques developed in \cite[\S 4]{De1} can not be easily applied. 
\end{rmk}

\begin{prop}\label{essentially finite and base change}
 Let $\Gamma$ be an affine gerbe over $k$. Let $V\in\Vect(\Gamma)$ be an object. Let $\langle V\rangle$ be the sub Tannakian category generated by $V$, and denote $\Delta$ the corresponding gerbe. Let $\beta:\Gamma\times_kl\arr \Gamma$ be the projection. Then $\langle \beta^*V\rangle\subseteq \Vect(\Gamma\times_kl)$ corresponds to $\Delta\times_kl$. In particular the canonical map
 \[
 \widehat{(\Gamma\times_k l)}\arr \widehat\Gamma \times_k l
 \]
 where $\widehat{-}$ denotes the profinite quotient, is an equivalence. In other words if $\shC$ is a $k$-Tannakian category then
 \[
 \EF(\shC)\otimes_k l\arr \EF(\shC\otimes_k l)
 \]
 is an equivalence.
\end{prop}

\begin{proof}
We have to show that $\langle \beta^*V\rangle= \langle V\rangle\otimes_kl$, that is an object $W\in\Vect(\Gamma\times_kl)$ belongs to $\langle\beta^*V\rangle$ if and only if $\beta_*W\in\langle V\rangle$. Since $\beta_*\beta^*V\in\langle V\rangle$, $\langle\beta^*V\rangle\subseteq \langle V\rangle\otimes_kl$. Composing with the forgetful functor, we see that $W\in \langle\beta^*V\rangle$ implies that $\beta_*W\in \langle V\rangle$. Conversely, if $\beta_*W\in \langle V\rangle$, then $\beta^*\beta_*W\in \langle \beta^*V\rangle$. Thus $W\in  \langle \beta^*V\rangle$ as there is a surjection $\beta^*\beta_*W\twoheadrightarrow W$. This finishes the first claim.

For the second claim we have to show that a vector bundle $V\in \Vect(\Gamma\times_kl)$ is essentially finite if and only if $\beta_*V\in \Vect(\Gamma)$ is essentially finite.  By the first claim $\beta_*V$ is essentially finite if and only if $\beta^*\beta_*V$ is essentially finite. Since we have a surjection $\beta^*\beta_*V\twoheadrightarrow V$ and $\beta^*\beta_*V$ is a finite direct sum of $V$, $V$ is essentially finite if and only if $\beta^*\beta_*V$ is so.
\end{proof}



\end{document}